\documentclass[12pt]{amsart}

\usepackage{fullpage}
\usepackage{amssymb,amsthm,amsmath,amsrefs,amsfonts}
\usepackage{dsfont}
\usepackage[all]{xy}

\numberwithin{equation}{section}

\newtheorem{theorem}{Theorem}
\newtheorem{corollary}{Corollary}
\newtheorem{lemma}{Lemma}
\newtheorem{proposition}{Proposition}
\newtheorem{definition}{Definition}

\newcommand{\K}{\mathcal{K}}
\newcommand{\N}{\mathbb{N}}

\newcommand{\R}{\mathbb{R}}
\newcommand{\C}{\mathbb{C}}
\newcommand{\Cu}{\mathrm{Cu}}
\newcommand{\CCu}{\mathbf{Cu}}

\newcommand{\id}{\mathrm{id}}

\newcommand{\Ad}{\mathrm{Ad}}
\newcommand{\Id}{\mathrm{Id}}
\newcommand{\M}{\mathrm{M}}

\newcommand{\rank}{\mathrm{rank}}
\newcommand{\Rank}{\mathrm{Rank}}
\newcommand{\diag}{\mathrm{diag}}

\title{A classification of inductive limits of splitting interval algebras}

\author{Luis Santiago}

\address{Departament de Matem\`atiques, Edifici C,
              Universitat Aut\`onoma de Barcelona, Bellaterra, Barcelona 08193, Spain.}
\email{santiago@mat.uab.cat}

\begin{document}

\begin{abstract}
It is shown that the Cuntz semigroup is a complete invariant for the C*-algebras that can be realized as an inductive limit of a sequence of finite direct sums of algebras of the form
\begin{align*}
\left\{f\in \M_m(\mathrm{C}[0,1]):f(0)\in \bigoplus^r_{i=1} \M_{p_i}(\C), f(1)\in \bigoplus^s_{i=1} \M_{q_i}(\C)\right\},
\end{align*}
where $p_1,p_2,\cdots, p_r$, and $q_1,q_2,\ldots, q_s$ are positive integers such that $\sum_{i=1}^rp_i=\sum_{i=1}^sq_i=m$.
\end{abstract}
\maketitle

\section{Introduction}
The results of this paper are a contribution to the classification program of C*-algebras.  The work on this program has been mostly concentrated on the classiﬁcation of simple C*-algebras, whereas the classification in the nonsimple case remains an emergent subject (compared to the body of work in the simple case). In recent years a new invariant have been successfully used to classify non-simple C*-algebras: the Cuntz semigroup (see \cite{Ciuperca-Elliott}, \cite{Ciuperca-Elliott-Santiago}). This semigroup is an analogue for positive elements of the semigroup of Murray-von Neumann equivalence classes of projections. Notably, the Cuntz semigroup contains a large amount of information about a given C*-algebra, its ideals and quotients. This makes the Cuntz semigroup suitable not only for the classification of simple C*-algebras but also for the classification of nonsimple ones.

In this paper we classify C*-algebras---no necessarily simple---that can be expressed as the inductive limit of a sequence of finite direct sums of splitting interval algebras (cf. Definition \ref{division}). The invariant used in the classification is the Cuntz semigroup. This result extends previous work of X. Jiang and H. Su where unital simple inductive limits of splitting interval algebras were classified using the Elliott invariant (see \cite{Jian-Su-Spli}). Also, the class of C*-algebras considered includes properly the class of AI C*-algebras, but goes beyond this. AI C*-algebras were classified by G. Elliott and A. Ciuperca in \cite{Ciuperca-Elliott}. The results of this paper can be considered as an extension of their classification result. However, the techniques used here differ from those developed by A. Ciuperca and  G. Elliott. Their result is based on a previous classification theorem of K. Thomsen in which a different invariant is used, while the classification theorems obtained in this paper are based on constructing an approximate intertwining of inductive limits of Cuntz semigroups. This notion of an approximate intertwining of two sequences of Cuntz semigroups is applicable in full generality (i.e., with no restriction on the C*-algebras under consideration). Therefore, it could be expected that it will be used to classify more general C*-algebras. (Possibly, even, certain simple C*-algebras.)

The main result of this paper states that the Cuntz semigroup functor classifies the $\ast$-homomorphisms between C*-algebras that are (stably isomorphic to) inductive limits of finite direct sums of splitting interval algebras:

\begin{theorem}\label{homomorphism}
Let $A$ and $B$ be C*-algebras that are (stably isomorphic to) inductive limits of finite direct sums of splitting interval algebras. Let $s_A$ and $s_B$ be strictly positive elements of $A$ and $B$, respectively. Suppose that there is a Cuntz semigroup morphism $\alpha\colon A\to B$ such that $\alpha [s_A]\le [s_B]$. It follows that there exists a $\ast$-homomorphism $\phi\colon A\to B$, unique up to approximate unitary equivalence, such that $\Cu(\phi)=\alpha$.
\end{theorem}

It follows from this theorem that the invariant consisting of the Cuntz semigroup together with a distinguished element of it classifies C*-algebras that are (stably isomorphic) to inductive limits of finite direct sums of splitting interval algebras: 

\begin{corollary}\label{isomorphism}
Let $A$ and $B$ be C*-algebras that are (stably isomorphic to) inductive limits of finite direct sums of splitting interval algebras. Let $s_A$ and $s_B$ be strictly positive elements of $A$ and $B$, respectively. Suppose that there is a Cuntz semigroup isomorphism $\alpha\colon A\to B$ such that $\alpha [s_A]= [s_B]$. It follows that there exists a $\ast$-isomorphism $\phi\colon A\to B$, unique up to approximate unitary equivalence, such that $\Cu(\phi)=\alpha$.
\end{corollary}

\section{Preliminary definitions and results}

\subsection{The Cuntz semigroup}\label{Cuntzsemigroup}
Let us recall the definition of the (stabilized) Cuntz semigroup of a C*-algebra in terms of the positive elements of the stabilization of the algebra. 

Let $A$ be a C*-algebra and let $a$ and $b$ be positive elements of $A$. Let us say that $a$ is Cuntz smaller than $b$, denoted by  $a\preccurlyeq b$ if there are elements $d_n\in A$, $n=1,2,\cdots$, such that $d_n^*bd_n\to a$. The elements $a$ and $b$ are called Cuntz equivalent, written $a\sim b$, if $a\preccurlyeq b$ and $b\preccurlyeq a$. It can be easily verified that $\preccurlyeq$ is a pre-order in the set of positive elements of $A$, and that $\sim$ is a equivalence relation.

Given a positive element $a\in A\otimes\K$ let us denote by $[a]$ the Cuntz equivalence class of $a$ ($\K$ denotes the algebra of compact operators on a separable Hilbert space). The Cuntz semigroup of $A$, denoted by $\Cu(A)$, is defined as the set of equivalence classes of positive elements of $A\otimes \K$ endowed with the order derived from the pre-order relation $\preccurlyeq$ (so that $[a]\le [b]$ if $a\preccurlyeq b$), and the addition operation
\begin{align*}
[a] + [b] := [a'+b'],
\end{align*}
where $a'$ and $b'$ are mutually orthogonal and Murray-von Neumann equivalent to $a$ and $b$, respectively (two positive elements $a, b\in A$ are said to be Murray-von Neumann equivalent, if there exists $x\in A$ such that $a=x^*x$ and $xx^*=b$). If $\phi\colon A\to B$ is a $\ast$-homomorphism from a C*-algebra $A$ to a C*-algebra $B$, then it induces an ordered semigroup morphism $\Cu(\phi)\colon \Cu(A)\to \Cu(B)$ defined on the Cuntz equivalence class of a positive element $a\in A\otimes \K$ by $\Cu(\phi)[a]=[(\phi\otimes\id)(a)]$, where $\Id\colon \K\to\K$ denotes the identity operator on $\K$.

It was shown in \cite{Coward-Elliott-Ivanescu} that $\Cu(\cdot)$ is a functor from the category of C*-algebras to certain category of ordered semigroups denoted by $\CCu$. An ordered semigroup $S$ is an object in the category $\CCu$ if it has a zero element and it satisfies that:

(i) every increasing sequence in $S$ has a supremum;

(ii) for every element $x\in S$ there is a sequence $(x_n)_{n\in\N}$ in $S$ with supremum $x$, and such that
\[
x_1\ll x_2\ll x_3\ll\cdots.
\]
Where $x\ll y$, if whenever $y\leq \sup y_n$ for an increasing sequence $(y_n)_{n\in \N}$, then eventually $x\leq y_n$;

(iii) the operation of passing to the supremum of an increasing sequence and the relation $\ll$ are
compatible with addition.

The maps in the category $\CCu$ are ordered semigroup maps preserving the zero element, suprema of
increasing sequences, and the relation $\ll$. In Theorem 2 of \cite{Coward-Elliott-Ivanescu} it is proved that sequential inductive limits exist in the category $\CCu$, and that the functor $\Cu$ preserves inductive limits of sequences. The following lemma follows by the construction of inductive limits given in the proof of this theorem:

\begin{proposition}\label{inductive_limits}
Let $(S_i,\alpha_{i,j})_{i,j\in \N}$, $\alpha_{i,j}\colon S_i\to S_j$, be an inductive system in the category $\CCu$. Then $(S,(\alpha_{i,\infty})_{i\in\N})$, $\alpha_{i,\infty}\colon S_i\to S$, is the inductive limit of this system if and only if:
\begin{itemize}
\item[(i)] for any $x\in S$ there are elements $x_i\in S_i$, $i=1,2,\cdots$, such that $\alpha_{i,i+1}(x_i)\le x_{i+1}$, and $\sup\alpha_{i,\infty}(x_i)=x$,

\item[(ii)] for any $x,y\in S_i$ such that $\alpha_{i,\infty}(x)\leq \alpha_{i,\infty}(y)$ and $x'\ll x$
there is $j$ such that $\alpha_{i,j}(x')\leq \alpha_{i,j}(x)$.
\end{itemize}
\end{proposition}

Let $S$ be a semigroup in the category $\CCu$. Recall that an element $x\in S$ is called compact if $x\ll x$. The following lemma states well known facts about compact elements. Its proof is an immediate consequence of Proposition \ref{inductive_limits}, whence we omit it.

\begin{lemma}\label{compact_elements}
Let $S$ be the inductive limit in the category $\CCu$ of a sequence of semigroups $(S_i)_{i\in \N}$ with connecting morphisms $\alpha_{i,j}\colon S_i\to S_j$. The following statements hold:
\begin{itemize}
\item[(i)] If $x,y\in S_i$ are such that $\alpha_{i,\infty}(x)\le\alpha_{i,\infty}(y)$ and $x$ is compact, then there exists $j\ge i$ such that $\alpha_{i,j}(x)\le\alpha_{i,j}(y)$;  

\item[(ii)] If $x,y\in S_i$ are compact elements such that $\alpha_{i,\infty}(x)=\alpha_{i,\infty}(y)$, then there exists $j\ge i$ such that $\alpha_{i,j}(x)=\alpha_{i,j}(y)$;

\item[(iii)] If $x\in S$ is a compact element, then there exists $j\ge 1$ and a compact element $x'\in S_j$ such that $\alpha_{k,\infty}(x')=x$.

\end{itemize}
\end{lemma}

\subsection{The Cuntz semigroup of splitting interval algebra}\label{Splitting}
The basic building blocks for the C*-algebras that are studied in this paper are certain C*-subalgebras of matrix algebras over the interval.

\begin{definition}\label{division}
A splitting interval algebra is a C*-algebra of the form:
\begin{align}\label{splitting}
\mathrm{S}_m[\overline{p},\overline{q},r,s]=\left\{f\in \M_m(\mathrm{C}[0,1]):f(0)\in \bigoplus^r_{i=1} \M_{p_i}(\C), f(1)\in \bigoplus^s_{i=1} \M_{q_i}(\C)\right\}
\end{align}
where $\overline{p}=(p_1,p_2,\cdots, p_r)$ and $\overline{q}=(q_1,q_2,\ldots, q_s)$ are tuples of positive integers such that $\sum_{i=1}^rp_i=\sum_{i=1}^sq_i=m$.
\end{definition}

Let us consider the splitting interval algebra $A=\mathrm{S}_m[\overline{p},\overline{q},r,s]$, and let us denote by $\mathrm{sp}(A)$ the spectrum of $A$. Then $\mathrm{sp}(A)$ is isomorphic to the set $\{0_1,\ldots, 0_r\}\cup (0,1)\cup \{1_1,\ldots, 1_s\}$ with the natural non-Hausdorff topology. That is, the topology generated by the sets $(t,1)\cup\{1_j\}$ and $\{0_i\}\cup (0,t)$, with $t\in (0,1)$, $1\le i\le r$, and $1\le j\le s$. The points in the spectrum are in one-to-one correspondence with the irreducible representations:
\begin{align}\label{irep}
\begin{aligned}
&\pi_{0_i}\colon \mathrm{S}_m[\overline{p},\overline{q},r,s]\to \M_{p_i}(\C),\quad \pi_{0_i}(f)=(f(0))_i,\\
&\pi_t\colon \mathrm{S}_m[\overline{p},\overline{q},r,s]\to \M_m(\C),\quad\pi_t(f)=f(t),\,t\in(0,1),\\
&\pi_{1_i}\colon \mathrm{S}_m[\overline{p},\overline{q},r,s]\to \M_{q_i}(\C),\quad \pi_{1_i}(f)=(f(1))_i.
\end{aligned}
\end{align}

Let $X$ be a topological space. Let us denote by $\mathrm{Lsc}(X, \N\cup\infty)$ the order semigroup of lower semicontinuous functions on $X$ with values in $\N\cup\infty$ (the order and addition are taking pointwise).

\begin{theorem}\label{Cuntzs_splitting}
Consider the splitting interval algebra $A=S_m[\overline{p},\overline{q},r,s]$. Then, the map
\begin{align}\label{cuntz_splitting}
\Rank\colon \Cu(A)\to \left\{f\in \mathrm{Lsc}(\mathrm{sp}(A), \N\cup \infty): 
\begin{array}{ll}
\lim_{t\to 0}f(t)\geq \sum_{i=1}^rf(0_i),\\
\lim_{t\to 1}f(t)\geq \sum_{i=1}^sf(1_i)
\end{array}
\right\},
\end{align}
given by $(\Rank[a])(t)=\rank((\pi_t\otimes \id_{\K})(a))$ is an isomorphism in the category $\CCu$.
\end{theorem}
\begin{proof}
The map $\Rank$ is clearly additive and order preserving.

Suppose that $a$ and $b$ are positive elements of $A\otimes \K$ such that $\Rank[a]\le \Rank[b]$. Then, it follows by Corollary 3.4 of \cite{Robert-Tikuisis} applied to the C*-algebra $A$ and to the Hilbert C*-modules $al _2(A)$ and $bl_2(A)$ that $a\preccurlyeq b$. Hence, the map $\Rank$ is an order embedding. Let us show that it is surjective. 

Let $f\in \mathrm{Lsc}(\mathrm{sp}(A), \N\cup \infty)$ be such that
\begin{align}\label{zero-one}
\lim_{t\to 0}f(t)\geq \sum_{i=1}^rf(0_i),\quad \lim_{t\to 1}f(t)\geq \sum_{i=1}^sf(1_i).
\end{align}
For each $n\ge 1$ set $f^{-1}((n,\infty])\cap (0,1)=U_n$. Then, since $f$ is lower semicontinuous the sets $U_n$, $n=1,2,\cdots$, are open and are such that $U_{n}\supseteq U_{n+1}$ for every $n\ge 1$. In addition,
\[
f(t)=\sum_{i=1}^\infty\mathds{1}_{U_n}(t),
\]
for all $t\in (0,1)$. (In the equation above $\mathds{1}_U$ denotes the characteristic function of the set $U$.)
Let us choose points 
\[
t_n\in\{0_1, 0_2,\cdots, 0_r\},\quad t'_n\in\{1_1, 1_2,\cdots, 1_s\},
\]
such that
\[
|\{n: t_n=0_i\}|=f(0_i), \quad |\{n: t'_n=1_j\}|=f(1_j),
\]
for every $1\le i\le r$ and $1\le j\le s$. Then, by \eqref{zero-one} and since $U_{n}\supseteq U_{n+1}$ for $n\ge 1$, the sets $V_n=\{t_n\}\cup U_n\cup \{t_n'\}$ are open in $\mathrm{sp}(A)$, and we have $f=\sum_{i=1}^\infty\mathds{1}_{V_n}$.

Let us show that for each $n\ge 1$ the characteristic function $\mathds{1}_{V_n}$ comes from a positive element of $A$. In order to show this, let us choose a rank one projection $P_n\in A$ and a continuous function $\lambda_n\colon [0,1]\to [0,1]$ such that $\pi_{t_n}(P_n)$ and $\pi_{t_n'}(P_n)$ have rank one, and such that the support of $\lambda_n$ is the open set $W_n\subseteq [0,1]$ given by
\begin{align*}
W_n=
\begin{cases}
U_n&\text{if } V_n=U_n,\\
\{0\}\cup U_n&\text{if } V_n\setminus U_n=\{0_i\}\text{ for some }1\le i\le r,\\
U_n\cup \{1\}&\text{if } V_n\setminus U_n=\{1_j\}\text{ for some } 1\le j\le s,\\
\{0\}\cup U_n\cup \{1\}&\text{if } V_n\setminus U_n=\{0_i, 1_j\}\text{ for some }1\le i\le r, 1\le j\le s.
\end{cases}
\end{align*}
Then, it follows that $\Rank[\lambda_nP_n]=\mathds{1}_{V_n}$. 

Let us choose positive contractions $a_n\in A\otimes \K$, $n=1,2,\cdots$, such that $a_ia_j=0$ for $i\neq j$, and such that $a_n$ is Cuntz equivalent to $\lambda_nP_n$ for all $n\ge 1$. Set $\sum_{i=1}^\infty a_i/2^i=a$. Then, for each $t\in \mathrm{sp}(A)$ we have
\begin{align*}
\Rank[a](t)&=\Rank\left[\lim_n\sum_{i=1}^n\frac{a_i}{2^i}\right](t)=\lim_n\sum_{i=1}^n\Rank[a_i](t)\\
&=\lim_n\sum_{i=1}^n[\lambda_i P_i](t)=\sup_n\sum_{i=1}^n\mathds{1}_{V_n}(t)\\
&=f(t)
\end{align*}
Therefore, $\Rank[a]=f$. This shows that the map $\Rank$ is surjective. Hence, it is an isomorphism in the category of ordered semigroups. Since $\Cu(A)$ belongs to the category $\CCu$ the semigroup in the right-hand side of equation \eqref{cuntz_splitting} belongs to the category $\CCu$, and the map $\Rank$ is a morphism in this category. 
\end{proof}

\subsection{Uniform structure}\label{Uniform-structure}
In this subsection we define a uniform structure in set of Cuntz semigroup morphisms between two Cuntz semigroup. This uniform structure was suggested to the author by L. Robert. Let us recall the definition of a uniform space.

\begin{definition}[\cite{Bourbaki}]
A uniform space $(X,\mathcal{U})$ is a set $X$ equipped with a nonempty family of subsets $\mathcal{U}$ of the Cartesian product $X\times X$ that satisfy the following axioms:
\begin{itemize}
\item[(i)] if $U$ is in $\mathcal{U}$, then $U$ contains the diagonal set $\{(x, x):x \in X \}$,

\item[(ii)] if $U$ is in $\mathcal{U}$ and $V$ is a subset of $X\times X$ which contains $U$, then $V$ is in $\mathcal{U}$,

\item[(iii)] if $U$ and $V$ are in $\mathcal{U}$, then $U \cap V$ is in $\mathcal{U}$,

\item[(iv)] if $U$ is in $\mathcal{U}$, then there exists $V$ in $\mathcal{U}$ such that
\[
V^2:=\{(x,z) : \exists y\text{ such that }(x,y),(y,z)\in V\},
\]
is a subset of $U$.

\item[(v)] if $U$ is in $\mathcal{U}$, then $U^{-1} := \{(y, x): (x, y)\in U\}$ is also in $\mathcal{U}$.
\end{itemize}
\noindent The family of subsets $\mathcal{U}$ is called a uniform structure on $X$ and its elements entourages.
\end{definition}

Let $S$ and $T$ be semigroups in the Cuntz category. Let us denote by $\mathrm{Mor}(S,T)$ the set of Cuntz semigroups morphisms from $S$ to $T$. For each finite subset $F$ of $S$ let us consider the set $U_F\subseteq \mathrm{Mor}(S,T)$  defined by
\begin{align*}
U_F=\{(\alpha, \beta) : \alpha(a)\le\beta(b), \beta(a)\le\alpha(b), \forall a,b\in F\text{ such that }a\ll b\}.
\end{align*} 
It follows that $U_{F\cup G}\subseteq U_F\cap U_G$, $U_F^{-1}=U_F$, and $U_{F'}^2\subseteq U_F$, where $F'$ is a finite subset of $S$ containing $F$ and satisfying that for all $a,b\in F$ such that $a\ll b$ there is $c\in F'$ such that $a\ll c\ll b$.
Therefore, the subsets $(U_F)_{F\subseteq S}$, $F$ finite, generate a uniform structure $\mathcal{U}_{S,T}$ on $\mathrm{Mor}(S,T)$.

Let $A=S_m[\overline{p},\overline{q},r,s]$ be a splitting interval algebra, and let $B$ be a C*-algebra. Let $x_{i,j}$, $1\le i\le r$, $1\le j\le s$, and $y_t$, $t\in (0,1)$, be the elements of $\Cu(A)$ defined by
\begin{align*}
x_{i,j}=\mathds{1}_{\{0_i\}\cup (0,1)\cup \{1_j\}},\quad y_t=\mathds{1}_{(t,1)\cup \{1_1\}},\quad y_1=0,
\end{align*}  
where $\mathds{1}_U$ denotes the characteristic function of the set $U$. For each $n\ge 1$, let us consider the finite set
\begin{align}\label{Fn}
F_n=\{x_{i,j} : 1\le i\le r, 1\le j\le s\}\cup \{y_{i/2^n}: 1\le i\le 2^n-1\}.
\end{align}
We have 
\begin{align}\label{uniform_1}
U_{F_{n+1}}\subseteq U_{F_n},\quad U_{F_{n+1}}^2\subseteq U_{F_n}.
\end{align}

\begin{proposition}\label{basis}
The entourages $(U_{F_n})_{n\in \N}$ form a basis for the uniform structure $\mathcal{U}_{\Cu(A), \Cu(B)}$.
\end{proposition}
\begin{proof} 
Since the entourages $(U_{F_n})_{n\in\N}$ satisfy \eqref{uniform_1} it is enough to show that given a finite subset $F$ of $\Cu(A)$ there is $n$ such that $U_{F_n}\subseteq U_F$. Let us first reduce to the case that $F$ consists of two characteristic functions of open intervals. In other words, characteristic functions of open sets of the form
$\{0_i\}\cup (0,1)\cup\{1_j\}$, $(t,1)\cup\{1_j\}$, $\{0_i\}\cup (0,t)$, or $(t,s)$,
with $t,s\in (0,1)$, $1\le i\le r$, and $1\le j\le s$.

We have $\bigcap_{a,b\in F,\, a\ll b}U_{\{a,b\}}\subseteq U_F$. Therefore, by \eqref{uniform_1} we may assume that $F$ consists of two elements. Let $F=\{a,b\}$ with $a\ll b$. Since $b$ is a lower semicontinuous function on $\mathrm{sp}(A)$, there exist open intervals $(V_n)_{n\in\N}$ in  $\mathrm{sp}(A)$ such that
\[
b=\sum_{n=1}^\infty \mathds{1}_{V_n}.
\]
For each $n\ge 1$ let us choose a sequence of open intervals $(V_{n,i})_{i\in\N}$ such that each $V_{n,i}$ has the same interval form as $V_n$,
\begin{align}\label{VK}
\bigcup_{i=1}^\infty V_{n,i}=V_n,\quad V_{n,1}\subseteq K_{n,1}\subseteq V_{n,2} \subseteq K_{n,2}\subseteq \ldots,
\end{align}
for some compact subsets $K_{n,i}$ of $\mathrm{sp}(A)$. Note that 
\begin{align}\label{Vn}
\mathds{1}_{V_{n,i}}\ll \mathds{1}_{V_{n,i+1}}.
\end{align}
Define
\[
b_i=\sum_{n=1}^i \mathds{1}_{V_{n,i}}.
\]
Then, by \eqref{VK} and \eqref{Vn} we have $b_i\ll b_{i+1}$ and $\sup b_i=b$. Since $a\ll b$ there exists $i\ge 1$ such that $a\le b_i\ll b_{i+1}\le b$, whence $U_{\{b_i,b_{i+1}\}}\in U_{\{a,b\}}$. Moreover, we have
\[
\bigcap_{n=1}^i U_{\{\mathds{1}_{V_{n,i}}, \mathds{1}_{V_{n,i+1}}\}}\subseteq U_{\{b_i,b_{i+1}\}}\subseteq U_{\{a,b\}}.
\]
This shows by \eqref{uniform_1} that $F$  may be taken to be a set formed by two characteristic functions of open intervals. 

There are a few cases to consider. Since the proof in all of them is similar we have chosen to give the proof in only two of them: $F=\{\mathds{1}_{\{0_i\}\cup (0,t)}, \mathds{1}_{\{0_i\}\cup (0,s)}\}$, with $t<s$, and $F=\{\mathds{1}_{(t_1,s_1)}, \mathds{1}_{(t_2,s_2)}\}$, with $t_2<t_1$ and $s_1<s_2$.

Suppose that $F=\{\mathds{1}_{\{0_i\}\cup (0,t)}, \mathds{1}_{\{0_i\}\cup (0,s)}\}$, with $t<s$.
Let us choose numbers $n$ and $i$ such that $t\le i/2^n< (i+1)/2^n\le s$, and let us show that $U_{F_n}\subseteq U_F$. 
By the choice of $n$ and $i$ we have
\[
\mathds{1}_{\{0_j\}\cup (0,t)}+ y_{i/2^n}\le x_{j,1}\ll x_{j,1}\le \mathds{1}_{\{0_j\}\cup (0,s)}+y_{(i+1)/2^n},
\]
for $1\le j\le r$. Let $(\alpha,\beta)\in U_{F_n}$. Then,
\[
\alpha(\mathds{1}_{\{0_j\}\cup (0,t)})+\alpha(y_{i/2^n})\le \alpha(x_{j,1})=\beta(x_{j,1})\le\beta(\mathds{1}_{\{0_j\}\cup (0,s)})+\beta(y_{(i+1)/2^n}),
\]
for $1\le j\le r$.
Using that $\beta(y_{(i+1)/2^n})\le \alpha(y_{i/2^n})$ in the equation above we conclude that
\[
\alpha(\mathds{1}_{\{0_j\}\cup (0,t)})+\alpha(y_{i/2^n})\le\alpha(x_{j,1})\ll \alpha(x_{j,1})\le \beta(\mathds{1}_{\{0_j\}\cup (0,s)})+\alpha(y_{i/2^n}).
\]
Hence, by Theorem 1 of \cite{Elliott-cancellation} (or by Proposition 4.2 of \cite{Rordam-Winter}), 
\[
\alpha(\mathds{1}_{\{0_j\}\cup (0,t)})\le \beta(\mathds{1}_{\{0_j\}\cup (0,t)}),
\]
for all $1\le j\le r$. By symmetry, $\beta(\mathds{1}_{\{0_j\}\cup (0,t)})\le \alpha(\mathds{1}_{\{0_j\}\cup (0,s)})$. Therefore, $U_{F_n}\subseteq U_F$.

Suppose that $F=\{\mathds{1}_{(t_1,s_1)}, \mathds{1}_{(t_2,s_2)}\}$, with $t_2<t_1$ and $s_1<s_2$.

Let us choose numbers $t'_1$ and $t'_2$ such that $t_2<t'_2<t'_1<t_1$. By the previous case, there exists $n$ such that 
\begin{align}\label{UFU}
U_{F_n}\subseteq U_{\left\{\mathds{1}_{\{0_1\}\cup (0,t'_2)}, \mathds{1}_{\{0_1\}\cup (0,t'_1)}\right\}}.
\end{align}
Moreover, we may take $n$ to be such that
\begin{align}\label{siis}
s_1\le i/2^n< (i+1)/2^n\le s_2,
\end{align}
holds for some $i\ge 1$. Let us show that $U_{F_n}\subseteq U_F$. 

Let $(\alpha, \beta)\in U_{F_n}$.
By \eqref{UFU} and \eqref{siis},
\[
\mathds{1}_{\{0_1\}\cup (0,t'_1)}+\mathds{1}_{(t_1,s_1)}+y_{i/2^n}\le x_{1,1}\ll x_{1,1}\le \mathds{1}_{\{0_1\}\cup (0,t'_2)}+\mathds{1}_{(t_2,s_2)}+y_{(i+1)/2^n}.
\]
Hence,
\begin{align}\label{alpha_beta}
\begin{aligned}
\alpha(\mathds{1}_{\{0_1\}\cup (0,t'_1)})+\alpha(\mathds{1}_{(t_1,s_1)})+\alpha(y_{i/2^n})&\le \alpha(x_{1,1})\\
&\ll\beta(x_{1,1})\\
&\le \beta(\mathds{1}_{\{0_1\}\cup (0,t'_2)})+\beta(\mathds{1}_{(t_2,s_2)})+\beta(y_{(i+1)/2^n}).
\end{aligned}
\end{align}
Using that $(\alpha, \beta)\in U_{F_n}$, and that $(\alpha, \beta)\in U_{\left\{\mathds{1}_{\{0_1\}\cup (0,t'_2)}, \mathds{1}_{\{0_1\}\cup (0,t'_1)}\right\}}$,
it follows that 
\[
\beta(y_{(i+1)/2^n})\le \alpha(y_{i/2^n}),\quad \beta(\mathds{1}_{\{0_1\}\cup (0,t'_2)})\le \alpha(\mathds{1}_{\{0_1\}\cup (0,t'_1)}).
\]
Hence, by \eqref{alpha_beta}
\begin{align*}
\alpha(\mathds{1}_{\{0_1\}\cup (0,t'_1)})+\alpha(\mathds{1}_{(t_1,s_1)})+\alpha(y_{i/2^n})&\le \alpha(x_{1,1})\\&\ll\alpha(x_{1,1})\\
&\le \alpha(\mathds{1}_{\{0_1\}\cup (0,t'_1)})+\beta(\mathds{1}_{(t_2,s_2)})+\alpha(y_{(i+1)/2^n}).
\end{align*}
By Theorem 1 of \cite{Elliott-cancellation} this implies that $\alpha(\mathds{1}_{(t_1,s_1)})\le \beta(\mathds{1}_{(t_2,s_2)})$. By symmetry, $\beta(\mathds{1}_{(t_1,s_1)})\le \alpha(\mathds{1}_{(t_2,s_2)})$, whence $(\alpha,\beta)\in U_{\{\mathds{1}_{(t_1,s_1)}, \mathds{1}_{(t_2,s_2)}\} }$.

\end{proof}

\subsection{Homomorphisms between splitting interval algebras}

Let $A=\mathrm{S}_m[\overline{p}, \overline{q}, r,s]$ be a splitting interval algebra. Let us consider a unital $\ast$-homomorphism $\phi\colon A\to \M_{m'}$. Then, $\phi$ is unitarily equivalent to a direct sum of irreducible representations of $A$. Therefore, there exist a unitary $U\in \M_{m'}$, tuples of positive integers $\overline{\nu}=(\nu_1, \nu_2, \cdots, \nu_r)$ and $\overline{\omega}=(\omega_1, \omega_2, \cdots, \omega_s)$, and a tuple of positive real numbers $\overline{\lambda}=(\lambda_1,\lambda_2,\cdots,\lambda_{\mu})$ with 
\[
0\le\lambda_1\le \lambda_2\le \cdots \le\lambda_\mu\le 1,
\]
such that
\begin{align}\label{phi_M_m}
\phi(f)=U^*\diag(\Lambda_{\overline{\nu}}(f), \Lambda_{\overline{\omega}}(f), f(\lambda_1),\cdots, f(\lambda_\mu))U,
\end{align}
where
\begin{align*}
\Lambda_{\overline{\nu}}=\bigoplus_{i=1}^r\pi_{0_i}\otimes 1_{\M_{\nu_i}},\quad \Lambda_{\overline{\omega}}=\bigoplus_{i=1}^r\pi_{1_i}\otimes 1_{\M_{\omega_i}}.
\end{align*}
By grouping the eigenvalues at $0$ and $1$, and using that 
\[
f(0)=\diag(\pi_{0_1}(f),\cdots, \pi_{0_r}(f)),\quad f(1)=\diag(\pi_{1_1}(f),\cdots, \pi_{1_r}(f)),
\]
we may always assume that $\nu_i=\omega_j=0$ for some $1\le i\le r$, and $1\le j\le s$. Note that if $\nu_i=\omega_j=0$ for some $1\le i\le r$,  and some $1\le j\le s$, then $\mu$, $\overline{\lambda}$, $\overline{\nu}$, and $\overline{\omega}$ are uniquely determined by $\phi$, and do not depend on the approximate unitary equivalence class of $\phi$. Therefore, we have a  map
\begin{align}\label{map1}
[\phi]\mapsto \left(\mu^\phi; \overline{\lambda^\phi}; \overline{\nu^\phi}; \overline{\omega^\phi}\right),
\end{align}
from the set of approximate unitary equivalent classes of unital $\ast$-homomorphisms from $A$ to $\M_{m'}$ to the set of tuples 
\[
\left(\mu; \overline{\lambda}; \overline{\nu}; \overline{\omega}\right) \in \N\times [0,1]^\mu\times\N^r\times \N^s,
\]
such that $0\le\lambda_1\le \lambda_2\le \cdots \le\lambda_\mu\le 1$, and such that $\nu_i=\omega_j=0$ for some $1\le i\le r$, and some $1\le j\le s$.

\begin{lemma}\label{Cuphi}
Let $A=\mathrm{S}_m[\overline{p}, \overline{q}, r,s]$ and let $\phi\colon A\to \M_{m'}$ be a $\ast$-homomorphism.
Then, the morphism $\Cu(\phi)\colon \Cu(A)\to \Cu(\M_{m'})$ is given by the formula:
\begin{align}\label{formulaCuphi}
\Cu(\phi)(f)=\sum_{i=1}^r\nu^\phi_i f(0_i)+\sum_{i=1}^s\omega^\phi_i f(1_i)+\sum_{i=1}^{\mu^\phi}f(\lambda^\phi_i),
\end{align}
where $\left(\mu^\phi; \overline{\lambda^\phi}; \overline{\nu^\phi}; \overline{\omega^\phi}\right)$ is the tuple associate to $\phi$ by the map \eqref{map1}. By convention in the equation above $f(0)=\sum_{i=1}^rf(0_i)$ and $f(1)=\sum_{i=1}^sf(1_i)$.
\end{lemma}
\begin{proof}
By the previous discussion and since the functor $\Cu(\cdot)$ is equal in $\ast$-homomorphisms that are unitarily equivalent we may assume that
\[
\phi(f)=\diag\left(\Lambda_{\overline{\nu^\phi}}(f), \Lambda_{\overline{\omega^\phi}}(f), f\left(\lambda^\phi_1\right),\cdots, f\left(\lambda^\phi_{\mu^\phi}\right)\right),
\]
or what it is the same
\begin{align}\label{sumadirecta}
\phi=\left(\bigoplus_{i=1}^r\pi_{0_i}\otimes 1_{\M_{\nu_i}}\right)\oplus\left(\bigoplus_{i=1}^r\pi_{1_i}\otimes 1_{\M_{\omega_i}}\right)\oplus\left(\bigoplus_{i=1}^{\mu^\phi}\delta_{\lambda^\phi_i}\right),
\end{align}
where $\delta_t\colon A\to \M_m$ denotes the evaluation map at the point $t$.

Let $a$ be a positive element of $A\otimes\K$. Then, we have
\begin{align*}
\Cu(\phi)[a]&=[(\phi\otimes\Id)(a)]=\rank((\phi\otimes\Id)(a))\\
&\stackrel{\eqref{sumadirecta}}{=}\sum_{i=1}^r\nu_i\rank((\pi_{0_i}\otimes\Id)(a))+\sum_{i=1}^s\omega_i\rank((\pi_{1_i}\otimes\Id)(a))+\sum_{i=1}^{\mu^\phi}\rank((\delta_{\lambda_i^\phi}\otimes \Id)(a))\\
&=\sum_{i=1}^r\nu^\phi_i [a](0_i)+\sum_{i=1}^s\omega^\phi_i [a](1_i)+\sum_{i=1}^{\mu^\phi}[a](\lambda^\phi_i).
\end{align*}
(We are using Theorem \ref{Cuntzs_splitting} to conclude the last step of the equation above.)
\end{proof}

Let $A=\mathrm{S}_m[\overline{p}, \overline{q}, r,s]$. For each $1\le i\le r$ and each $1\le j\le s$ let $P_{i,j}$ be a rank one projection such that $\pi_{0_i}(P_{i,j})$ and $\pi_{1_j}(P_{i,j})$ have rank one. Then, by Theorem \ref{Cuntzs_splitting} we have
\begin{align}\label{xyid}
x_{i,j}=[P_{i,j}], \quad y_t=[(\id-t)_+P_{1,1}]
\end{align}
for all $1\le i\le r$, $1\le j\le s$, and $t\in [0,1]$ (in the equation above $\id$ denotes the identity function of $\mathrm{C}[0,1]$). Consider the finite set 
\begin{align}\label{G}
G=\{P_{i,j}: 1\le i\le r,\, 1\le j\le s\}\cup\{(\id)P_{1,1}\}.
\end{align}
\begin{lemma}\label{lemma_uniqueness}
Let $n$ be a positive integer. The following statements hold:
\begin{itemize}
\item[(i)]  If $\phi,\psi\colon A\to \M_{m'}$ are $\ast$-homomorphisms such that
\begin{align}\label{ppdelta}
\|\phi(f)-\psi(f)\|<\frac{1}{2^n},
\end{align}
for all $f\in G$, then $(\Cu(\phi),\Cu(\psi))\in U_{F_n}$;

\item[(ii)] If $\phi,\psi\colon A\to \M_{m'}$ are $\ast$-homomorphisms such that $(\Cu(\phi),\Cu(\psi))\in U_{F_n}$, then $\mu^{\phi}=\mu^{\psi}$, $\overline{\nu^{\phi}}=\overline{\nu^{\psi}}$, $\overline{\omega^{\phi}}=\overline{\omega^{\psi}}$, and
\[
|\lambda_i^\phi-\lambda_i^\psi|<\frac{1}{2^{n-1}},
\]
for all $1\le i\le \mu^{\phi}$ ($=\mu^{\psi}$).

\end{itemize}
\end{lemma}
\begin{proof}
(i) Let $\phi,\psi\colon A\to \M_{m'}$ be $\ast$-homomorphisms satisfying \eqref{ppdelta}.  Let us show that $(\Cu(\phi), \Cu(\psi))\in U_{F_n}$. That is,
\begin{align*}
&\Cu(\phi)(x_{i,j})=\Cu(\psi)(x_{i,j}),\\
&\Cu(\phi)(y_{(k+1)/2^n})\le \Cu(\psi)(y_{k/2^n}),\\
&\Cu(\psi)(y_{(k+1)/2^n})\le \Cu(\phi)(y_{k/2^n}),
\end{align*}
 for all $1\le i\le r$, $1\le j\le s$, and $0\le k\le 2^n-1$.
 
Taking $f=P_{i,j}$ in equation \eqref{ppdelta} and applying Lemma 2.2 of \cite{Kirchberg-Rordam} to the elements $a=\phi(P_{i,j})$ and $b=\psi(P_{i,j})$ we have
\begin{align*}
\phi(P_{i,j})\preccurlyeq\psi(P_{i,j}),\quad \psi(P_{i,j})\preccurlyeq\phi(P_{i,j}).
\end{align*}
Hence, by \eqref{xyid}
\begin{align*}
\Cu(\phi)(x_{i,j})=\Cu(\phi)[P_{i,j}]=\Cu(\psi)[P_{i,j}]=\Cu(\psi)(x_{i,j}).
\end{align*}
Let $f=(\id)P_{1,1}$ in equation \eqref{ppdelta}. By Lemma 1 of \cite{Robert-Santiago} applied to the elements $a=\phi((\id)P_{1,1})$ and $b=\psi((\id)P_{1,1})$ 
\begin{align*}
&\phi((\id-(i+1)/2^n)_+P_{1,1})\preccurlyeq \psi((\id-i/2^n)_+P_{1,1})\\
&\psi((\id-(i+1)/2^n)_+P_{1,1})\preccurlyeq \phi((\id-i/2^n)_+P_{1,1}).
\end{align*}
Hence, by \eqref{xyid}
\begin{align*}
\Cu(\phi)(y_{(i+1)/2^n})\le\Cu(\psi)(y_{i/2^n}),\quad
\Cu(\psi)(y_{(i+1)/2^n})\le\Cu(\phi)(y_{i/2^n}).
\end{align*}

(ii) Let $\phi,\psi\colon A\to \M_{m'}$ be such that $(\Cu(\phi),\Cu(\psi))\in U_{F_n}$. By Lemma \ref{Cuphi} the morphisms $\Cu(\phi)$ and $\Cu(\psi)$ are given by:
\begin{align}\label{Cu_phi_psi}
\begin{aligned}
&\Cu(\phi)(f)=\sum_{i=1}^r\nu^\phi_i f(0_i)+\sum_{i=1}^s\omega^\phi_i f(1_i)+\sum_{i=1}^{\mu^\phi}f(\lambda^\phi_i),\\
&\Cu(\psi)(f)=\sum_{i=1}^r\nu^\psi_i f(0_i)+\sum_{i=1}^s\omega^\psi_i f(1_i)+\sum_{i=1}^{\mu^\psi}f(\lambda^\psi_i).
\end{aligned}
\end{align}
where by convention $f(0)=\sum_{i=1}^rf(0_i^A)$, and $f(1)=\sum_{i=1}^rf(1_i^A)$.
Evaluating $\Cu(\phi)$ and $\Cu(\psi)$ in the compact elements $x_{i,j}=\mathds{1}_{\{0_i\}\cup (0,1)\cup\{1_j\}}\in F_n$, $1\le i\le r$, $1\le j\le s$, and using that $(\Cu(\phi),\Cu(\psi))\in U_{F_n}$, we have
\begin{align}\label{no}
\nu^\phi_i+\omega^\phi_j+\mu^\phi=\Cu(\phi)(x_{i,j})=\Cu(\psi)(x_{i,j})=\nu^\psi_i+\omega^\psi_j+\mu^\psi.
\end{align}
By the choice of the tuples $\overline{\nu^\phi}$, $\overline{\omega^\phi}$, $\overline{\nu^\psi}$, and $\overline{\omega^\psi}$ there are indexes $i_1$, $i_2$, $j_1$, and $j_2$ such that $\nu^\phi_{i_1}=\omega^\phi_{j_1}=\nu^\psi_{i_2}=\omega^\psi_{j_2}=0$. Let us take $(i,j)$ equal to $(i_1,j_1)$, and $(i,j)$ equal to $(i_2, j_2)$  in equation \eqref{no}, then it follows that
\begin{align*}
&\mu^\phi\le\nu^\phi_{i_2}+\omega^\phi_{j_2}+\mu^\phi=\nu^\psi_{i_2}+\omega^\psi_{j_2}+\mu^\psi=\mu^\psi,\\
&\mu^\psi\le\nu^\psi_{i_1}+\omega^\psi_{j_1}+\mu^\psi=\nu^\phi_{i_1}+\omega^\phi_{j_1}+\mu^\phi=\mu^\phi.
\end{align*}
Therefore, $\mu^\phi=\mu^\psi$, and $\nu^\psi_{i_1}=\omega^\psi_{j_1}=\nu^\phi_{i_2}=\omega^\phi_{j_2}=0$. These identities together with equation \eqref{no} imply
\begin{align*}
&\nu^\phi_i+\mu^\phi=\nu^\phi_{i}+\omega^\phi_{j_1}+\mu^\phi=\nu^\psi_{i}+\omega^\psi_{j_1}+\mu^\psi
=\nu^\psi_{i}+\mu^\psi=\nu^\psi_{i}+\mu^\phi,\\
&\omega^\phi_j+\mu^\phi=\nu^\phi_{i_1}+\omega^\phi_{j}+\mu^\phi=\nu^\psi_{i_1}+\omega^\psi_{j}+\mu^\psi
=\omega^\psi_{j}+\mu^\psi=\omega^\psi_{j}+\mu^\phi.
\end{align*}
Therefore, $\nu^\phi_i=\nu^\psi_{i}$ and $\omega^\phi_j=\omega^\psi_j$, for all $1\le i\le r$ and $1\le j\le s$. It follows that $\overline{\nu^\phi}=\overline{\nu^\psi}$ and $\overline{\omega^\phi}=\overline{\omega^\psi}$.

Let us evaluate $\Cu(\phi)$ and $\Cu(\psi)$ at the elements $y_{i/2^n}=\mathds{1}_{(i/2^n,1)\cup \{1^A_1\}}\in F_n$, $0\le i\le 2^n-1$. By equation \eqref{Cu_phi_psi} we have that
\begin{align}\label{ppypp}
\begin{aligned}
&\Cu(\phi)(y_{i/2^n})=\omega_1^{\phi}+|\{j: \lambda^\phi_j> i/2^n\}|,\\
&\Cu(\psi)(y_{i/2^n})=\omega_1^{\psi}+|\{j: \lambda^\psi_j> i/2^n\}|,
\end{aligned}
\end{align} 
for every $0\le i\le 2^n-1$. (In the equation above we are using the notation $|S|$ to denote the number of elements of the set $S$.) Using that $y_{(i+1)/2^n}\ll y_{i/2^n}$, and that $(\Cu(\phi),\Cu(\psi))\in U_{F_n}$ the morphisms $\Cu(\phi)$ and $\Cu(\psi)$ satisfy that
\[
\Cu(\phi)(y_{(i+1)/2^n})\le \Cu(\psi)(y_{i/2^n}),\quad \Cu(\psi)(y_{(i+1)/2^n})\le \Cu(\phi)(y_{i/2^n}),
\]
for $0\le i\le 2^n-2$. Hence, by \eqref{ppypp}
\begin{align*}
|\{j: \lambda^\phi_j> (i+1)/2^n\}|\le |\{j: \lambda^\psi_j> i/2^n\}|,\\
|\{j: \lambda^\psi_j> (i+1)/2^n\}|\le |\{j: \lambda^\phi_j> i/2^n\}|,
\end{align*}
for all $0\le i\le 2^n-2$. These inequalities together with
\begin{align*}
\lambda^\phi_1\le \lambda^\phi_2\le\cdots\le \lambda^\phi_{\mu^\phi},\quad
\lambda^\psi_1\le \lambda^\psi_2\le\cdots\le \lambda^\psi_{\mu^\psi},
\end{align*}
imply that
\begin{align*}
|\lambda_i^\phi-\lambda_i^\psi|\le\frac{1}{2^{n-1}},
\end{align*}
for all $1\le i\le \mu^\phi$ ($=\mu^\psi$).
\end{proof}

\begin{proposition}
Let $A=\mathrm{S}_m[\overline{p}, \overline{q}, r,s]$ and $B=\M_{m'}(\mathrm{C}[0,1])$. Let $\phi\colon A\to B$ be a unital $\ast$-homomorphism. Then, there exists a unital $\ast$-homomorphism $\psi\colon A\to B$ that is approximately unitarily equivalent to $\phi$, and that is such that
\begin{align*}
\psi(f)=\diag\left(\Lambda_{\overline{\nu}}(f), \Lambda_{\overline{\omega}}(f), f\circ\lambda_1,\cdots, f\circ\lambda_\mu)\right),
\end{align*}
where the tuples $\overline{\nu}\in\N^r$ and $\overline{\omega}\in \N^s$ have a component equal to zero, and the functions $\lambda_i\colon [0,1]\to [0,1]$ are continuous and such that
$0\le\lambda_1\le \lambda_2\le \cdots \le\lambda_\mu\le 1$.
\end{proposition}
\begin{proof}
Let $\phi\colon A\to B$ be a unital $\ast$-homomorphism. For each $t\in [0,1]$ let us consider the composition $\delta_t\circ\phi\colon A\to \M_{m'}$, where $\delta_t\colon B\to \M_{m'}$ denotes the evaluation map at the point $t$. Using the map \eqref{map1} we can associate to $\delta_t\circ\phi$ a tuple $\left(\mu(t), \overline{\lambda(t)}, \overline{\nu(t)}, \overline{\omega(t)}\right)$ such that
\[
0\le\lambda_1(t)\le\lambda_2(t)\le \cdots\le \lambda_{\mu(t)}\le 1,
\]
such that $\overline{\nu(t)}$ and $\overline{\omega(t)}$ have a component equal to zero, and such that
\begin{align*}
\phi(f)(t)=(\delta_t\circ\phi)(f)=U_t^*\diag\left(\Lambda_{\overline{\nu(t)}}(f), \Lambda_{\overline{\omega(t)}}(f), f(\lambda_1(t)),\cdots, f(\lambda_{\mu(t)}(t))\right)U_t,
\end{align*}
for some unitary $U_t\in \M_{m'}$. Let us show that $\mu(t)$, $\overline{\nu(t)}$, and $\overline{\omega(t)}$ do not depend on $t$, and that for each $1\le i\le\mu=\mu(t)$ the function $t\in [0,1]\to \lambda_i(t)$ is continuous.

Let $n\in \N$ be fixed and let $G\subset A$ be the finite set defined in \eqref{G}. Since the functions $\phi(f)$, $f\in G$, are uniformly continuous there exists $\epsilon>0$ such that $|t-t'|<\epsilon$ implies 
\begin{align*}
\|(\delta_t\circ\phi)(f)-(\delta_{t'}\circ\phi)(f)\|<\frac{1}{2^n},
\end{align*}
for every $f\in G$. Hence, by (i) of Lemma \ref{lemma_uniqueness} we have $(\Cu(\delta_t\circ\phi), \Cu(\delta_{t'}\circ\phi))\in U_{F_n}$. This implies by the second part of the same lemma that $\mu(t)=\mu(t')$, $\overline{\nu(t)}=\overline{\nu(t')}$, $\overline{\omega(t)}=\overline{\omega(t')}$, and
\begin{align*}
|\lambda_i(t)-\lambda_i(t')|\le\frac{1}{2^{n-1}},
\end{align*}
for every $t,t'\in [0,1]$ with $|t-t'|<\epsilon$.
It follows that $\mu(t)$, $\overline{\nu(t)}$, and $\overline{\omega(t)}$ are constant, and that for each $1\le i\le\mu=\mu(t)$ the function $t\in [0,1]\to \lambda_i(t)$ is continuous. To simplify notation let us write $\mu$, $\overline{\nu}$, and $\overline{\omega}$ instead of $\mu(t)$, $\overline{\nu(t)}$, and $\overline{\omega(t)}$.

For each $t\in [0,1]$ and each $f\in A$ we have
\begin{align}\label{standard_0}
\phi(f)(t)=U_t^*\diag\left(\Lambda_{\overline{\nu}}(f), \Lambda_{\overline{\omega}}(f), f(\lambda_1(t)),\cdots, f(\lambda_\mu(t))\right)U_t.
\end{align}
Consider the $\ast$-homomorphism $\psi\colon A\to B$ defined by
\begin{align*}
\psi(f)=\diag\left(\Lambda_{\overline{\nu}}(f), \Lambda_{\overline{\omega}}(f), f\circ\lambda_1,\cdots, f\circ\lambda_\mu)\right).
\end{align*}
Then, $\rank((\phi\otimes \id_{\K})(f))=\rank((\psi\otimes \id_{\K})(f))$ for all $f\in A\otimes \K$. Hence, by Theorem \ref{Cuntzs_splitting} we have $\Cu(\phi)=\Cu(\psi)$. This implies that $\mathrm{K}_0(\phi)=\mathrm{K}_0(\psi)$ since $A$ is stably finite. By the choice of $\psi$, $\phi(f)(t)$ and $\psi(f)(t)$ have the same eigenvalues for each $t\in [0,1]$ and each $f\in A$. Therefore, by Proposition 7.3 of \cite{Su} the $\ast$-homomorphisms $\phi$ and $\psi$ are approximately unitarily equivalent.
\end{proof}

\begin{theorem}\label{homomorphism1}
Let $A=\mathrm{S}_m[\overline{p}, \overline{q}, r,s]$ and $B=\mathrm{S}_{m'}[\overline{p'}, \overline{q'}, r',s']$. Let $\phi\colon A\to B$ be a unital $\ast$-homomorphism. Then, there exists a unital $\ast$-homomorphism $\psi\colon A\to B$ that is approximately unitarily equivalent to $\phi$, and that is such that
\begin{itemize}
\item[(i)]
\begin{align*}
\psi(f)=U^*\diag\left(\Lambda_{\overline{\nu}}(f), \Lambda_{\overline{\omega}}(f), f\circ\lambda_1,\cdots, f\circ\lambda_\mu)\right)U,
\end{align*}
where $U$ is a unitary of $\M_{m'}(\mathrm{C}[0,1])$ such that $U(0)$ and $U(1)$ are permutation matrices, the tuples $\overline{\nu}\in\N^r$ and $\overline{\omega}\in \N^s$ have a component equal to zero, and the functions $\lambda_i\colon [0,1]\to [0,1]$ are continuous and such that
\[
0\le\lambda_1\le \lambda_2\le \cdots \le\lambda_\mu\le 1;
\]

\item[(ii)] for each $t\in \mathrm{sp}(B)\setminus (0,1)$,

\begin{align*}
(\pi_t\circ\psi)(f)=\diag\left(\Lambda_{\overline{\nu_t}}(f), \Lambda_{\overline{\omega_t}}(f), f(\lambda_{t,1}),\cdots, f(\lambda_{t,\mu_t})\right), 
\end{align*}
where the tuples $\overline{\nu_t}\in \N^r$ and $\overline{\omega_t}\in \N^s$ have a component equal to zero, and
\[
0\le\lambda_{t,1}\le \lambda_{t,2}\le \cdots \le\lambda_{t,\mu_t}\le 1
\]
($\pi_t$ denotes the irreducible representation of $B$ given in \eqref{irep}.)
\end{itemize}
\end{theorem}
\begin{proof}
Let us consider $\phi$ as a $\ast$-homomorphism from $A$ to $\M_{m'}(\mathrm{C}[0,1])$. Then, by the previous discussion there are a family of unitaries $U_t\in \M_{m'}$, $t\in [0,1]$, tuples $\overline{\nu}\in\N^r$ and $\overline{\omega}\in \N^s$ that have a component equal to zero, and continuous functions $\lambda_i\colon [0,1]\to [0,1]$, $1\le i\le\mu$, with
\[
0\le\lambda_1\le \lambda_2\le \cdots \le\lambda_\mu\le 1,
\]
such that \eqref{standard_0} holds for $\phi$. 

For each $t\in \mathrm{sp}(B)\setminus (0,1)$ define
\[
\left(\overline{\nu_t}; \overline{\omega_t}; \mu_t;\lambda_{t,1},\lambda_{t,2},\cdots, \lambda_{t, \mu_t}\right):=\left(\overline{\nu^{\pi_t\circ\phi}}; \overline{\omega^{\pi_t\circ\phi}}; \mu^{\pi_t\circ\phi};\lambda^{\pi_t\circ\phi}_1, \lambda^{\pi_t\circ\phi}_2\cdots, \lambda^{\pi_t\circ\phi}_{\mu^{\pi_t\circ\phi}} \right),
\]
where the right hand side of the equation above is the tuple associated to the $\ast$-homomorphism $\pi_t\circ \phi$ by the map \eqref{map1}. It is clear that $\pi_t\circ \phi$ is unitarily equivalent to the $\ast$-homomorphism
\[
f\in A\longmapsto \diag\left(\Lambda_{\overline{\nu_t}}(f), \Lambda_{\overline{\omega_t}}(f), f(\lambda_{t,1}),\cdots, f(\lambda_{t,\mu_t})\right).
\]
Since for every $f\in A$ we have
\[
\phi(f)(0)=\bigoplus_{i=1}^{r'}(\pi_{0_i}\circ \phi)(f), \quad \phi(f)(1)=\bigoplus_{i=1}^{s'}(\pi_{1_i}\circ \phi)(f),
\]
there are permutation matrices $S_0, S_1\in \M_{m'}$ such that
\begin{align*}
&S_0^*\diag\left(\Lambda_{\overline{\nu}}(f), \Lambda_{\overline{\omega}}(f), f(\lambda_1(0)),\cdots, f(\lambda_\mu(0)))\right)S_0\\
&=\bigoplus_{i=1}^{r'} \diag\left(\Lambda_{\overline{\nu_{0_i}}}(f), \Lambda_{\overline{\omega_{0_i}}}(f), f(\lambda_{0_i, 1}),\cdots, f(\lambda_{0_i, \mu_{0_i}})\right),\\
&S_1^*\diag\left(\Lambda_{\overline{\nu}}(f), \Lambda_{\overline{\omega}}(f), f(\lambda_1(1)),\cdots, f(\lambda_\mu(1)))\right)S_1\\
&=\bigoplus_{i=1}^{r'} \diag\left(\Lambda_{\overline{\nu_{1_i}}}(f), \Lambda_{\overline{\omega_{1_i}}}(f), f(\lambda_{1_i,1}),\cdots, f(\lambda_{1_i,\mu_{1_i}})\right),
\end{align*}
for all $f\in A$. Choose a unitary $U\in \M_{m'}(\mathrm{C}[0,1])$ such that $U(0)=S_0$ and $U(1)=S_1$. Let us define the $\ast$-homomorphism $\psi\colon A\to B$ by
\[
\psi(f)=U^*\diag\left(\Lambda_{\overline{\nu}}(f), \Lambda_{\overline{\omega}}(f), f\circ\lambda_1,\cdots, f\circ\lambda_\mu)\right)U.
 \]
Then, by construction $\psi$ satisfies conditions (i) and (ii) of the theorem. It is left to prove that $\psi$ is approximately unitarily equivalent to $\phi$. By Proposition 7.3 of \cite{Su} it is enough to show that $\Cu(\phi)=\Cu(\psi)$ since in the current case this implies that $\mathrm{K}_0(\phi)=\mathrm{K}_0(\psi)$. By the construction of $\psi$, for each $t\in \mathrm{sp}(B)$ the $\ast$-homomorphisms $\pi_t\circ\phi$ and $\pi_t\circ\psi$ are unitarily equivalent. Hence, 
\[
\rank(((\pi_t\otimes \id_{\mathcal{K}})\circ(\phi\otimes \id_{\mathcal{K}}))(f))=\rank(((\pi_t\otimes \id_{\mathcal{K}})\circ(\psi\otimes \id_{\mathcal{K}}))(f))
\]
for all $f\in (A\otimes \mathcal{K})^+$. Therefore, $\Cu(\phi)=\Cu(\psi)$ by Theorem \ref{Cuntzs_splitting}.
\end{proof}

\begin{definition}\label{standard}
Let $A$ and $B$ be splitting interval algebras with $A=\mathrm{S}_m[\overline{p}, \overline{q}, r,s]$ and $B=\mathrm{S}_{m'}[\overline{p'}, \overline{q'}, r',s']$. Let us say that a $\ast$-homomorphism $\phi\colon A\to B$ has standard form if satisfies parts \rm{(i)} and \rm{(ii)} of Theorem \ref{homomorphism1}.
\end{definition}

\section{Approximate uniqueness}

Let $A$ be a C*-algebra. We say that a class of C*-algebras $\mathcal{B}$ has the approximate uniqueness property with respect to $A$ if given $\epsilon>0$ and a finite subset $G\subset A$ there exists a finite subset $F\subset \Cu(A)$ such that; if $B\in \mathcal{B}$, and $\phi,\psi\colon A\to B$ are $\ast$-homomorphisms such that 
\[
(\Cu(\phi), \Cu(\psi))\in U_F,
\]
then there exists a unitary $u\in \widetilde{B}$ such that
\begin{align}\label{au_2}
\|\phi(a)-u^*\psi(a)u\|<\epsilon,
\end{align}
for all $a\in G$. If the class $\mathcal{B}$ consists of a single C*-algebra $B$ we will say that the C*-algebra $B$ has the approximate uniqueness property with respect to $A$.

\begin{theorem}\label{Uniqueness}
Let $A$ be a splitting interval algebra, and let $B$ be a sequential inductive limit of finite direct sums of splitting interval algebras. Then $B$ has the approximate uniqueness property with respect to $A$.
\end{theorem}

\begin{lemma}\label{au_unitalhomomorphisms}
Let $\mathcal{B}$ be the class of splitting interval algebras. Then $\mathcal{B}$ has the approximate uniqueness property with respect to a splitting interval algebra $A$ (in the category of C*-algebras with arbitrary $\ast$-homomorphisms), if and only if, it has the approximate uniqueness property with respect to $A$ in the category of unital C*-algebras with unital $\ast$-homomorphisms.
\end{lemma}
\begin{proof}
The implication ($\Rightarrow$) clearly holds. Let us prove the reverse implication $(\Leftarrow)$. Assume that $\mathcal{B}$ has the approximate uniqueness property with respect to a splitting interval algebra $A$ in the category of unital C*-algebras with unital $\ast$-homomorphisms.

Given $\epsilon>0$ and a finite subset $G\subset A$ let $F$ be the finite subset of $\Cu(A)$ given by the approximate uniqueness property of the class $\mathcal{B}$ with respect to $A$ in the category of unital C*-algebras with unital $\ast$-homomorphisms. Set $F\cup \{[1_A]\}=F'$, where $[1_A]$ denotes the Cuntz equivalence class of the unit of $A$. Let us show that the set $F'$ satisfies the conditions in the definition of the approximate uniqueness property of the class $\mathcal{B}$ with respect to $A$ for $\epsilon$ and $G$. 

Consider a splitting interval algebra $B$ and $\ast$-homomorphisms $\phi,\psi\colon A\to B$ such that $(\phi,\psi)\in U_{F'}$. By the definition of the entourage $U_{F'}$, and using that $[1_A]$ is a compact element of $\Cu(A)$ we have that $\Cu(\phi)[1_A]=\Cu(\psi)[1_A]$. This implies that the projections $\phi(1_A)$ and $\psi(1_A)$ are Cuntz equivalent, and hence unitarily equivalent since $B$ has stable rank one. Let $v\in B$ be a unitary such that $\phi(1_A)=v^*\psi(1_A)v$. Consider the $\ast$-homomorphism $\psi'=\Ad(v)\circ \psi$. It follows that $\phi(1_A)=\psi'(1_A)$, and 
\begin{align}\label{pp'}
(\phi,\psi')\in U_{F}.
\end{align}
Let us denote by $p$ the projection $\phi(1_A)$. If we restrict the codomain of the $\ast$-homomorphisms $\phi$ and $\psi'$ to the hereditary subalgebra $pBp$, then the new $\ast$-homomorphisms---which we will denote again by $\phi$ and $\psi'$---are unital and satisfy condition \eqref{pp'}. The hereditary subalgebra $pBp$ is isomorphic to a splitting interval algebra since $p$, being a projection of a splitting interval algebra, can be diagonalize in the matrix algebra over $C[0,1]$ that contains $B$ unitally. Therefore, by the choice of the set $F$ there exists a unitary $w\in pBp$ such that
\[
\|\phi(a)-w^*\psi'(a)w\|<\epsilon,
\]
for all $a\in G$. Set $v(w+1-p)=u$. Then, $u\in B$ is a unitary and 
\[
\|\phi(a)-u^*\psi(a)u\|<\epsilon,
\]
for all $a\in G$.
\end{proof}

\begin{lemma}\label{au_directsums}
Let $A$ be a C*-algebra and let $\mathcal{B}$ be a class of C*-algebras that has the approximate uniqueness property with respect to $A$. Then the class of C*-algebras $\mathcal{B'}$ consisting of the finite direct sums of C*-algebras in $\mathcal{B}$ satisfies the approximate uniqueness property with respect to $A$.
\end{lemma}
\begin{proof}
Given $\epsilon>0$ and $G\subset A$ finite, let $F\subset \Cu(A)$ be the finite subset given by the approximate uniqueness property of the class $\mathcal{B}$. Let $(B_i)_{i=1}^n$ be C*-algebras in $\mathcal{B}$, and let $\phi,\psi\colon A\to B=\bigoplus_{i=1}^nB_i$ be *-homomorphisms. It follows that
\[
\phi=(\phi_1,\phi_2,\cdots, \phi_n),\quad \psi=(\psi_1,\psi_2,\cdots, \psi_n),
\]
with $\phi_i,\psi_i\colon A\to B_i$.

Suppose that $(\Cu(\phi), \Cu(\psi))\in U_{F}$. Then, since $\Cu(B)$ is naturally isomorphic to the direct sum $\bigoplus_{i=1}^n\Cu(B_i)$. We have that $(\Cu(\phi_i), \Cu(\psi_i))\in U_{F}$, for every $1\le i\le n$. Hence, by the choice of the set $F\subset \Cu(A)$, there are unitaries $u_i\in \widetilde{B_i}$, $i=1,2,\cdots, n$, such that
\[
\|\phi_i(a)-u_i\psi(a)u_i\|<\epsilon,
\]
for all $a\in G$. Moreover, the unitaries $u_i$ may be taken such that $u_i-1_{\widetilde{B_i}}\in B_i$. Set 
\[
1_{\widetilde{B}}+\sum_{i=1}^n(u_i-1_{\widetilde{B_i}})=u.
\]
Then $u$ is a unitary in $\widetilde{B}$, and
\[
\|\phi(a)-u^*\psi(a)u\|<\epsilon.
\]
This shows that $\mathcal{B'}$ has the approximate uniqueness property with respect to $A$. 
\end{proof}

\begin{proposition}\label{au_limits}
Let $A$ be the universal C*-algebra generated by a finite number of elements satisfying a stable, finite, and bounded set of relations. Let $B$ be C*-algebra that is the inductive limit of the sequence of C*-algebras
\begin{align*}
\xymatrix{
B_1\stackrel{\rho_1}\longrightarrow B_2\stackrel{\rho_2}\longrightarrow B_3\stackrel{\rho_3}\longrightarrow\cdots.
}
\end{align*}
If the class of C*-algebras $\mathcal{B}=\{B_i\}_{i=1}^\infty$ has the approximate uniqueness property with respect to $A$, then the limit algebra $B$ has the approximate uniqueness property with respect to $A$.
\end{proposition}
\begin{proof}
Given $\epsilon>0$ and $G\subset A$ finite, let $F\subset \Cu(A)$ be the finite subset given by the approximate uniqueness property of the class $\mathcal{B}=(B_i)_{i=1}^\infty$. Let $F'\subset \Cu(A)$ be a finite subset such that $F\subseteq F'$, and such that for every $x,y\in F$ with $x\ll y$ there are $z_1,z_2,z_3\in F'$ such that $x\ll z_1\ll z_2\ll z_3\ll y$.  In order to prove that the proposition holds it is enough to show that given $\ast$-homomorphisms $\phi,\psi\colon A\to B$ such that $(\phi, \psi)\in U_{F'}$,
there exists a unitary $u\in \widetilde{B}$ such that
\[
\|\phi(a)-u^*\psi(a)u\|<3\epsilon,
\]
for all $a\in G$.

Let $\phi,\psi\colon A\to B$ be $\ast$-homomorphisms such that 
\begin{align}\label{phipsi}
(\phi, \psi)\in U_{F'}.
\end{align}
By hypothesis the C*-algebra $A$ is the universal C*-algebra generated by a finite number of elements satisfying a stable, finite, and bounded set of relations. Hence, by Proposition 14.1.2 of \cite{LoringBook} there is an integer $k\ge 1$, and $\ast$-homomorphisms $\phi_i,\psi_i\colon A\to B_i$, $i=k,k+1,\cdots$, such that
\begin{align*}
\lim_{i\to \infty} (\rho_{i,\infty}\circ\phi_i)(a)=\phi(a),\quad \lim_{i\to \infty}(\rho_{i,\infty}\circ\psi_i)(a)=\psi(a),
\end{align*} 
for all $a\in A$, where $\rho_{i,\infty}\colon B_i\to B$, $i=1,2,\cdots$, are the $\ast$-homomorphisms induced by the inductive limit. This implies that
\begin{align}\label{limppk}
\begin{aligned}
&\lim_{i\to \infty} ((\rho_{i,\infty}\otimes \id_{\K})\circ(\phi_i\otimes \id_{\K}))(a)=(\phi\otimes \id_{\K})(a),\\
&\lim_{i\to \infty} ((\rho_{i,\infty}\otimes \id_{\K})\circ(\psi_i\otimes \id_{\K}))(a)=(\psi\otimes \id_{\K})(a),
\end{aligned}
\end{align} 
for all $a\in A\otimes \K$, where $\id_{\K}\colon \K\to \K$ denotes the identity operator $\K$. Let us choose $0<\epsilon'<\epsilon/3$, and a finite subset $G'$ of $A\otimes\K$ such that $G\subseteq G'$ (here $G$ is identified with its image under the canonical inclusion $i_A\colon A\to A\otimes K$ given by $i_A(a)=a\otimes e_{1,1}$, $a\in A$), and such that for each $x,y\in F'$ with $x\ll y$ there is $a\in G'$ such that 
\begin{align}\label{xay}
x\ll [(a-\epsilon')_+]\ll [a]\ll y.
\end{align}
Let us see that this is in fact possible. Let $x,y\in F'$ be such that $x\ll y$. Since every element of $\Cu(A)$ is the supremum of a rapidly increasing sequence of elements of $\Cu(A)$ there exists $z\in \Cu(A)$ such that $x\ll z\ll y$. Let $a\in (A\otimes \K)^+$ be such that $[a]=z$. We have 
\[
x\ll [a]=\sup_{\delta>0}[(a-\delta)_+].
\]
Hence, by the definition of the relation $\ll$ there is $\delta>0$ such that $x\ll [(a-\epsilon')_+]\ll [a]$ for all $0<\epsilon'<\delta$. Since $F'$ is finite we can choose $0<\epsilon'<\epsilon$ such that for each $x,y\in F'$ with $x\ll y$ there is $a\in (A\otimes \K)^+$ such that \eqref{xay} holds. The set $G'$ is taking to be the union of the set $G$ and the set of elements $a$ associated to each pair of elements $x,y\in F'$ with $x\ll y$.

By \eqref{limppk} there exists $i\ge k$ such that
\begin{align}\label{inepp}
\begin{aligned}
\|((\rho_{i,\infty}\otimes \id_{\K})\circ(\phi_i\otimes \id_{\K}))(a)-(\phi\otimes \id_{\K})(a)\|<\epsilon',\\
\|((\rho_{i,\infty}\otimes \id_{\K})\circ(\psi_i\otimes \id_{\K}))(a)-(\psi\otimes \id_{\K})(a)\|<\epsilon',
\end{aligned}
\end{align}
for all $a\in G'$. Hence, by Lemma 2.2 of \cite{Kirchberg-Rordam} 
\begin{align}\label{rppa}
\begin{aligned}
\Cu(\rho_{i,\infty}\circ \phi_i)[(a-\epsilon')_+]&= [((\rho_{i,\infty}\circ\phi_i)\otimes \id_{\K}))((a-\epsilon')_+)]\\
&\le [(\phi\otimes \id_{\K})(a)]\\
&=\Cu(\phi)[a],
\end{aligned}
\end{align}
for all $a\in G'$.

Let $x,y\in F'$ be such that $x\ll y$. Then by the choice of $\epsilon'$ and the finite set $G'$ there exists $a\in G'$ such that \eqref{xay} holds. Hence, using \eqref{rppa} we have
\begin{align*}
\Cu(\rho_{i,\infty}\circ \phi_i)(x)\le \Cu(\rho_{i,\infty}\circ \psi_i)[(a-\epsilon')_+]\le \Cu(\phi)[a]\le \Cu(\phi)(y).
\end{align*}
By symmetry,
\begin{align*}
\Cu(\phi)(x)\le\Cu(\rho_{i,\infty}\circ \phi_i)(y).
\end{align*}
Since the preceding inequalities holds for all $x,y\in F'$ with $x\ll y$ we have
\begin{align}\label{riphi}
(\Cu(\rho_{i,\infty}\circ \phi_i), \Cu(\phi))\in U_{F'}.
\end{align}
Similarly, 
\begin{align}\label{ripsi}
(\Cu(\rho_{i,\infty}\circ \psi_i), \Cu(\psi))\in U_{F'}.
\end{align}
Let us show that there is $n\ge i$ such that $(\Cu(\rho_{i,n}\circ \phi_i), \Cu(\rho_{i,n}\circ \phi_i))\in U_{F}$. Let $x,y\in F$ be such that $x\ll y$. By the choice of the finite set $F'$ the elements $x$ and $y$ are in $F'$, and there are $z_1,z_2,z_3\in F'$ such that $x\ll z_1\ll z_2\ll z_3\ll y$. Using \eqref{phipsi}, \eqref{riphi}, and \eqref{ripsi} we have
\[
\Cu(\rho_{i,\infty}\circ \phi_i)(x)\ll \Cu(\rho_{i,\infty}\circ \phi_i)(z_1)\le \Cu(\phi)(z_2)\le \Cu(\psi))(z_3)\le \Cu(\rho_{i,\infty}\circ \psi_i)(y).
\]
By symmetry,
\[
\Cu(\rho_{i,\infty}\circ \psi_i)(x)\ll \Cu(\rho_{i,\infty}\circ \psi_i)(z_1)\le \Cu(\rho_{i,\infty}\circ \phi_i)(y).
\]
This implies by (ii) of Proposition \ref{inductive_limits} that there exists $N\ge i$ such that
\begin{align}\label{xypp}
\begin{aligned}
\Cu(\rho_{i,n}\circ \phi_i)(x)\le \Cu(\rho_{i,n}\circ \psi_i)(y),\\
\Cu(\rho_{i,n}\circ \psi_i)(x)\le \Cu(\rho_{i,n}\circ \phi_i)(y),
\end{aligned}
\end{align}
for all $n\ge N$, where $\rho_{i,n}$ denotes the composition $\rho_{n-1}\circ\rho_{n-2}\circ\cdots\circ\rho_i$. Therefore, since $F$ is finite we can choose $n\ge i$ such that \eqref{xypp} holds simultaneously for all $x,y\in F$ with $x\ll y$. This shows that $(\Cu(\rho_{i,n}\circ \phi_i), \Cu(\rho_{i,n}\circ \phi_i))\in U_{F}$. Hence, there exists a unitary $v$ in the unitization of the C*-algebra $B_n$ such that
\[
\|(\rho_{i,n}\circ \phi_i)(a)-v^*(\rho_{i,n}\circ \psi_i)(a)v\|< \epsilon,
\]
for all $a\in G$. Let $u\in \widetilde{B}$ be the image of $v$ by the unitization of the map $\rho_{n,\infty}$. It follows that $u$ is a unitary and that
\[
\|(\rho_{i,\infty}\circ \phi_i)(a)-u^*(\rho_{i,\infty}\circ \psi_i)(a)u\|< \epsilon,
\] 
for all $a\in G$. Using the triangle inequality, the inequalities in Equation \eqref{inepp}, and the preceding inequality we have
\begin{align*}
\|\phi(a)-u^*\psi(a)u\|&\le\|\phi(a)-(\rho_{i,\infty}\circ \phi_i)(a)\|+\|(\rho_{i,\infty}\circ \phi_i)(a)-u^*(\rho_{i,\infty}\circ \psi_i)(a)u\|+\\
&+\|(\rho_{i,\infty}\circ \psi_i)(a)-\psi(a)\|\\
&< \epsilon'+\epsilon+\epsilon'< 3\epsilon.
\end{align*}
This concludes the proof of the proposition.
\end{proof}

\begin{proof}[Proof of Theorem \ref{Uniqueness}]
Let $A=S_m[\overline{p},\overline{q}, r, s]$. By Lemma \ref{au_unitalhomomorphisms}, Lemma \ref{au_directsums}, and Proposition \ref{au_limits} it is sufficient to prove that the class of splitting interval algebras has the approximate uniqueness property with respect to $A$ in the category of unital C*-algebras with unital $\ast$-homomorphisms. In order to do this, it is enough to show that given $\epsilon>0$ and $G\subset A$ finite we can choose a positive integer $n$ such that;  if $B$ is a splitting interval algebra, and $\phi, \psi\colon A\to B$ are unital $\ast$-homomorphisms such that $(\Cu(\phi), \Cu(\psi))\in U_{F_n}$, then there exists a unitary $u\in B$ such that
\begin{align}\label{unitpp}
\|\phi(f)-u^*\psi(f)u\|<3\epsilon,
\end{align}
where $F_n$ denotes the finite subset of $\Cu(A)$ defined in \eqref{Fn}. Let us show that if $n$ is chosen such that for every $f\in G$ and every $x_1,x_2\in [0,1]$ with $|x_1-x_2|\le 1/2^{n-1}$,
\[
\|f(x_1)-f(x_2)\|< \epsilon,
\]
then the previous statement holds.

Let $B=S_{m'}[\overline{p'}, \overline{q'}, r', s']$, and let $\phi, \psi\colon A\to B$ be unital $\ast$-homomorphisms such that 
\begin{align}\label{Upp}
(\Cu(\phi), \Cu(\psi))\in U_{F_n}.
\end{align}
By Theorem \ref{homomorphism1} and using that the functor $\Cu(\cdot)$ is equal in $\ast$-homomorphisms that are approximately unitarily equivalent we may assume that $\phi$ and $\psi$ have standard form:
\begin{align*}
&\phi(f)=(U^\phi)^*\diag\left(\Lambda_{\overline{\nu^\phi}}(f), \Lambda_{\overline{\omega^\phi}}(f), f\circ\lambda^\phi_1,\cdots, f\circ\lambda^\phi_{\mu^\phi})\right)U^\phi,\\
&(\pi_t\circ\phi)(f)=\diag\left(\Lambda_{\overline{\nu^\phi_t}}(f), \Lambda_{\overline{\omega^\phi_t}}(f), f(\lambda^\phi_{t,1}),\cdots, f(\lambda^\phi_{t,\mu^\phi_t})\right),\text{ if }t\in \mathrm{sp}(B)\setminus (0,1),
\end{align*}
\begin{align*}
&\psi(f)=(U^\psi)^*\diag\left(\Lambda_{\overline{\nu^\psi}}(f), \Lambda_{\overline{\omega^\psi}}(f), f\circ\lambda^\psi_1,\cdots, f\circ\lambda^\psi_{\mu^\psi})\right)U^\psi,\\
&(\pi_t\circ\psi)(f)=\diag\left(\Lambda_{\overline{\nu^\psi_t}}(f), \Lambda_{\overline{\omega^\psi_t}}(f), f(\lambda^\psi_{t,1}),\cdots, f(\lambda^\psi_{t,\mu^\psi_t})\right),\text{ if }t\in \mathrm{sp}(B)\setminus (0,1),
\end{align*}

By \eqref{Upp} we have that $(\Cu(\pi_t\circ\phi), \Cu(\pi_t\circ\psi))\in U_{F_n}$ for every $t\in \mathrm{sp}(B)$. Hence, by (ii) of Lemma \ref{lemma_uniqueness} applied to $\pi_t\circ\phi$ and $\pi_t\circ\psi$ we have
\begin{align*}
(\overline{\nu^\phi}, \overline{\omega^\phi}, \mu^\phi)=(\overline{\nu^\psi}, \overline{\omega^\psi}, \mu^\psi),\quad (\overline{\nu^\phi_t}, \overline{\omega^\phi_t}, \mu^\phi_t)=(\overline{\nu^\psi_t}, \overline{\omega^\psi_t}, \mu^\psi_t),
\end{align*}
for all $t\in \mathrm{sp}(B)\setminus (0,1)$, and 
\begin{align}\label{ll0}
\sup_{t\in [0,1]}|\lambda_i^\phi(t)-\lambda_i^\psi(t)|\le\frac{1}{2^{n-1}},
\end{align}
for $1\le i\le \mu^\phi (=\mu^\psi)$.

Let us choose continuous functions $\rho_i\colon [0,1]\to [0,1]$, $i=1,2\cdots, \mu^\phi (=\mu^\psi)$, such that
\begin{align}\label{llrho}
|\lambda_i^\phi(t)-\rho_i(t)|\le\frac{1}{2^{n-1}},\quad |\lambda_i^\psi(t)-\rho_i(t)|\le\frac{1}{2^{n-1}},
\end{align}
for $1\le i\le \mu^\phi (=\mu^\psi)$, such that $\rho_i(0)$ is equal to $0$ or $1$ if either $\lambda_i^\phi(0)$ or $\lambda_i^\psi(0)$ is equal to $0$ or $1$, and such that $\rho_i(1)$ is equal to $0$ or $1$ if either $\lambda_i^\phi(1)$ or $\lambda_i^\psi(1)$ is equal to $0$ or $1$. Let $\phi'$ and $\psi'$ be the $\ast$-homomorphisms defined by:
\begin{align*}
&\phi'(f)=(U^\phi)^*\diag\left(\Lambda_{\overline{\nu^\phi}}(f), \Lambda_{\overline{\omega^\phi}}(f), f\circ\rho_1,\cdots, f\circ\rho_{\mu^\phi})\right)U^\phi,\\
&\psi'(f)=(U^\psi)^*\diag\left(\Lambda_{\overline{\nu^\psi}}(f), \Lambda_{\overline{\omega^\psi}}(f), f\circ\rho_1,\cdots, f\circ\rho_{\mu^\psi})\right)U^\psi,
\end{align*}
for $f\in A$. Note that by the choice of the maps $\rho_i$ the images of the $\phi'$ and $\psi'$ are contain in $B$, and 
\begin{align*}
&(\overline{\nu^{\phi'}}, \overline{\omega^{\phi'}}, \mu^{\phi'})=(\overline{\nu^\phi}, \overline{\omega^\phi}, \mu^\phi)=(\overline{\nu^\psi}, \overline{\omega^\psi}, \mu^\psi)=(\overline{\nu^{\psi'}}, \overline{\omega^{\psi'}}, \mu^{\psi'}),\\
&(\overline{\nu^{\phi'}_t}, \overline{\omega^{\phi'}_t}, \mu^{\phi'}_t)=(\overline{\nu^\phi_t}, \overline{\omega^\phi_t}, \mu^\phi_t)=(\overline{\nu^\psi_t}, \overline{\omega^\psi_t}, \mu^\psi_t)=(\overline{\nu^{\psi'}_t}, \overline{\omega^{\psi'}_t}, \mu^{\psi'}_t),
\end{align*}
for all $t\in \mathrm{sp}(B)\setminus (0,1)$. This implies---as in the proof of Theorem \ref{homomorphism1}---that $\Cu(\phi')=\Cu(\psi')$, whence $\mathrm{K}_0(\phi')=\mathrm{K}_0(\psi')$. By Proposition 7.3 of \cite{Su} the $\ast$-homomorphisms $\phi$ and $\psi$ are approximately unitarly equivalent. It follows that there exists a unitary $u\in B$ such that
\begin{align}\label{p'p'}
\|\phi'(f)-u^*\psi(f)u\|<\epsilon,
\end{align}
for all $f\in G$. By the choice of $n$ and by \eqref{llrho} we have
\begin{align*}
\sup_{t\in [0,1]}|f(\lambda^\phi_i(t))-f(\rho_i(t))|<\epsilon,\quad \sup_{t\in [0,1]}|f(\lambda^\psi_i(t))-f(\rho_i(t))|<\epsilon,
\end{align*}
for all $f\in G$, and $1\le i\le \mu^\phi$.
Hence,
\begin{align*}
\|\phi(f)-\phi'(f)\|<\epsilon, \quad \|\psi(f)-\psi'(f)\|<\epsilon.
\end{align*}
These inequalities together with \eqref{p'p'} imply that
\[
\|\phi(f)-u^*\psi(f)u\|<3\epsilon,
\]
for all $f\in G$.
\end{proof}

\section{Approximate lifting}

\begin{theorem}\label{Existence}
Let $A$ be a splitting interval algebra and let $B$ be an inductive limit of finite direct sums of splitting interval algebras. Let $s_B$ be a strictly positive element of $B$, and let $F\subset \Cu(A)$ be a finite subset. If $\alpha\colon \Cu(A)\to \Cu(B)$ is a morphism in the category $\CCu$ such that $\alpha[1_A]\le [s_B]$, then there exists a $\ast$-homomorphism $\phi\colon A\to B$ such that 
\[
(\alpha, \Cu(\phi))\in U_F.
\]
\end{theorem}

Let $A=\mathrm{S}_m[\overline{p},\overline{q},r,s]$. For each $n=1,2,\cdots$, let us consider the finite subset $F_n\subset\Cu(A)$ given by
\begin{align*}
F_n=\{x_{i,j} : 1\le i\le r, 1\le j\le s\}\cup \{y_{i/2^n}: 1\le i\le 2^n\}.
\end{align*}
where 
\begin{align*}
&x_{i,j}=\mathds{1}_{\{0_i\}\cup (0,1)\cup \{1_j\}},\\
& y_t=\mathds{1}_{(t,1)\cup \{1_1\}},\quad y_1=0,
\end{align*}  
(Here we are identifying $\Cu(A)$ with the semigroup in the right hand side of Equation \eqref{cuntz_splitting}.)

\begin{lemma}\label{app-lif1}
Let $A=\mathrm{S}_m[\overline{p},\overline{q},r,s]$ and let $B$ be either a direct sum matrix algebras or a matrix algebra over the algebra of continuous functions on the interval $[0,1]$. Let $\alpha\colon F_n\subset \Cu(A)\to \Cu(B)$ be a map such that for every $1\le i,i'\le r$, $1\le j,j'\le s$, and $1\le k\le 2^n-1$
\begin{itemize}
\item[(i)] $\alpha(x_{i,j})$ is compact;

\item[(ii)] $\alpha(x_{i,1})\ll\alpha(y_{k/2^n})+\alpha(x_{i,j})$;

\item[(iii)] $\alpha(y_{(k+1)/2^n})\ll \alpha(y_{k/2^n})$;

\item[(iv)] $\alpha(y_0)\le \alpha(x_{i,1})$;

\item[(v)] $\alpha(x_{i,j})+\alpha(x_{i',j'})=\alpha(x_{i,j'})+\alpha(x_{i',j})$;

\item[(vi)] $\alpha(x_{i,j})\le \alpha[1_A]=[1_B]$;

\item[(vii)]
\begin{align*}
\sum_{i=1}^rp_i\alpha(x_{i,j'})+\sum_{j=1}^sq_j\alpha(x_{i',j})=\alpha[1_A]+(2m-1)\alpha(x_{i',j'}).
\end{align*}
\end{itemize}
Then, there exits a unital $\ast$-homomorphism $\phi\colon A\to B$ such that
\begin{align*}
(\alpha,\Cu(\phi))\in U_{F_n}.
\end{align*} 
\end{lemma}
\begin{proof}
Let $B=\M_{m'}(\mathrm{C}[0,1])$ and let $\alpha\colon F_n\subset\Cu(A)\to \Cu(B)$ be as in the statement of the lemma. Let us choose indexes $i'$ and $j'$ such that
\[
\alpha(x_{i',j'})(0)\le \alpha(x_{i,j})(0)
\]
for all $1\le i\le r$, and $1\le j\le s$. Note that the inequality above also holds if the point $0$ is replaced by any point of the closed interval $[0,1]$ since by condition (i) of the lemma $\alpha(x_{i,j})$ and $\alpha(x_{i',j'})$ are compact elements of $\Cu(B)$ ($\cong \mathrm{Lsc}([0,1],\N\cup\infty))$, and hence, they are the constant functions. By condition (ii) of the lemma \begin{align*}
\alpha(x_{i',1})\ll\alpha(y_{k/2^n})+\alpha(x_{i',j'}),
\end{align*}
for every $0\le k\le 2^n-1$. Since $\alpha(x_{i',1})$ is a compact element of $\Cu(B)$ and $B$ has stable rank one, $\alpha(x_{i',1})$ is the class of a projection of $B\otimes \K$. It follows by Proposition 2.2 of \cite{Perera-Toms} that for each $1\le k\le 2^n-1$ there exists an element $z_{k/2^n}$ such that
\begin{align}\label{zxyx}
z_{k/2^n} +\alpha(x_{i',1})=\alpha(y_{k/2^n})+\alpha(x_{i',j'}).
\end{align}
By conditions (iii) and (iv) we have
\[
\alpha(y_0)\ll \alpha(x_{i',1}),\quad \alpha(y_{(k+1)/2^n})\ll \alpha(y_{k/2^n}).
\]
Hence,
\begin{align*}
z_0+\alpha(x_{i',1})&\stackrel{\eqref{zxyx}}{=}\alpha(y_0)+\alpha(x_{i',j'})\\
&\ll \alpha(x_{i',1})+\alpha(x_{i',j'}),
\end{align*}
\begin{align*}
z_{(k+1)/2^n} +\alpha(x_{i',1})&\stackrel{\eqref{zxyx}}{=}\alpha(y_{(k+1)/2^n})+\alpha(x_{i',j'})\\
&\ll \alpha(y_{k/2^n})+\alpha(x_{i',j'})\\
&\stackrel{\eqref{zxyx}}{=} z_{k/2^n} +\alpha(x_{i',1}).
\end{align*}
More briefly,
\[
z_0+\alpha(x_{i',1})\ll \alpha(x_{i',j'})+\alpha(x_{i',1}),\quad z_{(k+1)/2^n} +\alpha(f_{i',1})\ll z_{k/2^n} +\alpha(f_{i',1}).
\]
These inequalities imply by Theorem 4.3 of \cite{Rordam-Winter} that $z_{(k+1)/2^n}\ll z_{k/2^n}$, and $z_0\ll\alpha(x_{i',j'})$. Thus, the elements $(z_{k/2^n})_{k=0}^{2^n-1}$ are such that
\begin{align}\label{azs}
z_{(2^n-1)/2^n}\ll z_{(2^n-2)/2^n}\ll \cdots \ll z_{1/2^n}\ll z_0\ll \alpha(x_{i',j'}).
\end{align}
By Lemma 4 of \cite{Robert-Santiago} there exists a positive contraction $a\in B\otimes \K$ such that
\begin{align}\label{zzaz}
z_0=[a],\quad z_{(k+1)/2^n}\ll [(a-k/2^n)]\ll z_{k/2^n},
\end{align}
for every $0\le k\le 2^n-1$.

By condition (vi) we have  $\alpha(x_{i',j'})\le \alpha[1_A]=[1_B]$. Hence, since $B$ has stable rank one and $\alpha(x_{i',j'})$ is compact there exists a projection $p\in B$ such that $[p]=\alpha(x_{i',j'})$. Moreover, since $p$ is unitarily equivalent to a trivial projection we may assume that $p=\diag(1,\cdots,1,0,\cdots,0)$. By the choice of $a$ and by \eqref{azs},
\begin{align*}
[a]=z_0\ll \alpha(x_{i',j'})=[p].
\end{align*}
Hence, by Theorem 3 of \cite{Coward-Elliott-Ivanescu} (see also Proposition 1 of \cite{Ciuperca-Elliott-Santiago}) there is a positive element $b\in pBp$ such that $a=x^*x$, and $b=xx^*$ for some element $x\in B$. By (i) of Lemma 5 of \cite{Ciuperca-Elliott-Santiago} this implies that 
\[
[(b-k/2^n)_+]=[(a-k/2^n)_+],
\]
for every $0\le k\le 2^n-1$. Since $b\in p\M_{m'}(\mathrm{C}[0,1])p=\M_{\rank(p)}(\mathrm{C}[0,1])$, by Lemma 1.1 and Theorem 1.2 of \cite{Thomsen} there are continuous functions $\lambda_{i}\colon [0,1]\to [0,1]$, $i=1,2,\cdots,\rank(p)$, such that $\lambda_1\ge \lambda_2\ge \cdots\ge \lambda_{\rank(p)}$, and such that the positive element 
\[
c=\diag(\lambda_1,\lambda_2,\cdots, \lambda_{\rank(p)}, 0,\cdots,0)\in \M_{m'}(\mathrm{C}[0,1]),
\] 
is approximately unitarly equivalent to $b$. This implies that for each $t\in \R^+$ the positive elements $(a-t)_+$ and $(b-t)_+$ are approximately unitarly equivalent, and hence Cuntz equivalent. Therefore, we have
\begin{align}\label{cba}
[(c-k/2^n)_+]=[(b-k/2^n)_+]=[(a-k/2^n)_+],
\end{align}
for all $0\le k\le 2^n-1$.

For each $1\le i\le r$, and $1\le j\le s$ define
\begin{align}\label{munuomega}
\begin{aligned}
&\mu:=\rank(p),\\
&\nu_i:=\alpha(x_{i,j'})(0)-\alpha(x_{i',j'})(0),\\
&\omega_j:=\alpha(x_{i',j})(0)-\alpha(x_{i',j'})(0).
\end{aligned}
\end{align}
By condition (vii) of the lemma
\begin{align*}
&\sum_{i=1}^rp_i\nu_i+\sum_{j=1}^sq_j\omega_j+\mu=\\
&=\sum_{i=1}^rp_i(\alpha(x_{i,j'})(0)-\alpha(x_{i',j'})(0))+\sum_{j=1}^sq_j(\alpha(x_{i',j})(0)-\alpha(x_{i',j'})(0))+\alpha(x_{i',j'})(0)\\
&=\sum_{i=1}^rp_i\alpha(x_{i,j'})(0)+\sum_{j=1}^sq_j\alpha(x_{i',j})(0)-(2m-1)\alpha(x_{i',j'})(0)\\
&=\alpha[1_A](0)\\
&=[1_B](0)\\
&=m'.
\end{align*}
Let us consider the map
\begin{align*}
\phi(f)=\diag(\Lambda_{\overline{\nu}}(f),\Lambda_{\overline{\omega}}(f),f\circ\lambda_1, f\circ\lambda_2,\cdots, f\circ\lambda_{\mu}).
\end{align*}
It is clear that $\phi$ is a $\ast$-homomorphism. Also, by the previous computation $\phi$ maps $A$ into $\M_{m'}(\mathrm{C}[0,1])$. Let us show that $(\alpha, \Cu(\phi)\in U_{F_n}$.

By Lemma \ref{Cuphi} the values of the morphism $\Cu(\phi)$ at the elements $x_{i,j}$ and $y_{k/2^n}$ are given by the formulas
\begin{align}\label{xmnylw}
\Cu(\phi)(x_{i,j})(t)=\mu+\nu_i+\omega_j,\quad \Cu(\phi)(y_{k/2^n})(t)=|\{i:\lambda_i(t)>k/2^n\}|+w_1.
\end{align}
We have
\begin{align*}
\Cu(\phi)(x_{i,j})(t)&\stackrel{\eqref{xmnylw}}{=}\mu+\nu_i+\omega_j\\
&\stackrel{\eqref{munuomega}}{=}\alpha(x_{i',j'})(0)+\alpha(x_{i,j'})(0)-\alpha(x_{i',j'})(0)+\alpha(x_{i',j})(0)-\alpha(x_{i',j'})(0)\\
&\stackrel{\mathrm{(v)}}{=}\alpha(x_{i,j'})(0)+\alpha(x_{i',j})(0)-\alpha(x_{i',j'})(0)\\
&=\alpha(x_{i,j})(0)+\alpha(x_{i',j'})(0)-\alpha(x_{i',j'})(0)\\
&=\alpha(x_{i,j})(0)=\alpha(x_{i,j})(t),
\end{align*}
\begin{align*}
\Cu(\phi)(y_{(k+1)/2^n})(t)&\le\Cu(\phi)(y_{k/2^n})(t)\stackrel{\eqref{xmnylw}}{=}|\{i:\lambda_i(t)>k/2^n\}|+w_1\\
&\stackrel{\eqref{munuomega}}{=}\rank((c-k/2^n)_+(t))+\alpha(x_{i',1})(0)-\alpha(x_{i'j'})(0)\\
&=[(c-k/2^n)_+](t)+\alpha(x_{i',1})(t)-\alpha(x_{i',j'})(t)\\
&\stackrel{\eqref{cba},\eqref{zzaz}}{\le}z_{k/2^n}(t)+\alpha(x_{i',1})(t)-\alpha(x_{i',j'})(t)\\
&\stackrel{\eqref{zxyx}}{=}\alpha(y_{k/2^n})(t)+\alpha(x_{i',j'})(t)-\alpha(x_{i',1})(t)+\alpha(x_{i',1})(t)-\alpha(x_{i',j'})(t)\\
&=\alpha(y_{k/2^n})(t),\\
\alpha(y_{(k+1)/2^n})(t)&\stackrel{\eqref{zxyx}}{=} z_{(k+1)/2^n}(t)+\alpha(x_{i',1})(t)-\alpha(x_{i',j'})(t)\\
&\stackrel{\eqref{cba},\eqref{zzaz}}{\le} [(c-k/2^n)_+](t)+\alpha(x_{i',1})(0)-\alpha(x_{i',j'})(0)\\
&=\rank((c-k/2^n)_+(t))+\alpha(x_{i',1})(0)-\alpha(x_{i',j'})(0)\\
&=|\{i:\lambda_i(t)>k/2^n\}|+w_1\\
&\stackrel{\eqref{xmnylw}}{=}\Cu(\phi)(y_{k/2^n})(t),
\end{align*}
for every $t\in [0,1]$ and $0\le k\le 2^n-1$. Hence,
\begin{align*}
\alpha(x_{i,j})=\Cu(\phi)(x_{i,j}),\quad \alpha(y_{(k+1)/2^n})\le \Cu(\phi)(y_{k/2^n}),\quad \Cu(\phi)(y_{(k+1)/2^n})\le \alpha(y_{k/2^n}),
\end{align*}
for all $1\le i\le r$, $1\le j\le s$, and $0\le k\le 2^n-1$. This shows that $(\alpha, \Cu(\phi))\in U_{F_n}$.
A repetition of the previous arguments shows that the lemma holds in the case that $B$ is a direct sum of matrix algebras. 
\end{proof}

\begin{lemma}\label{app-lif2}
Let $A$ and $B$ be splitting interval algebras, with $A=\mathrm{S}_m[\overline{p},\overline{q},r,s]$. Let $\alpha\colon F_n\subset \Cu(A)\to \Cu(B)$ be a map satisfying conditions {\rm (i)} to {\rm (vii)} of the preceding lemma. Then, there exists a $\ast$-homomorphism $\phi\colon A\to B$ such that $(\alpha, \Cu(\phi))\in U_{F_{n-2}}$.
\end{lemma}
\begin{proof}
Let $B=\mathrm{S}_{m'}[\overline{p'},\overline{q'},r',s']$ and let $\alpha\colon F_n\subset \Cu(A)\to \Cu(B)$ be map satisfying conditions (i) to (vii) of Lemma \ref{app-lif1}. Let us consider the Cuntz semigroup morphisms
\begin{align*}
\Cu(\mathrm{i}_B)\circ\alpha &\colon \Cu(A)\to \Cu(\M_{m'}(\mathrm{C}[0,1])),\\
\Cu(\delta_0)\circ\alpha &\colon\Cu(A)\to \Cu\left(\bigoplus_{i=1}^r\M_{p'_i}\right),\\
\Cu(\delta_1)\circ\alpha &\colon\Cu(A)\to \Cu\left(\bigoplus_{j=1}^s\M_{q'_j}\right),
\end{align*}
where $\mathrm{i}_B\colon B\to \M_{m'}(\mathrm{C}[0,1])$ denotes the inclusion map, and $\delta_0\colon B\to \bigoplus_{i=1}^r\M_{p'_i}$ and $\delta_1\colon B\to \bigoplus_{j=1}^s\M_{q'_j}$ denote the evaluation maps at $0$ and $1$, respectively. By Lemma \ref{app-lif1} there exist $\ast$-homomorphism
\[
\phi\colon A\to \M_{m'}(\mathrm{C}[0,1]),\quad \phi_0\colon A\to \bigoplus_{i=1}^r\M_{p'_i},\quad \phi_1\colon A\to \bigoplus_{j=1}^s\M_{q'_j},
\]
such that $\phi$ has standard form, $\phi_0$ and $\phi_1$ have diagonal form, and
\begin{align}
&(\Cu(\mathrm{i}_B)\circ\alpha,\Cu(\phi))\in U_{F_n}, \label{UUU1}\\
&(\Cu(\delta_0)\circ\alpha,\Cu(\phi_0))\in U_{F_n},\label{UUU2}\\
&(\Cu(\delta_1)\circ\alpha,\Cu(\phi_1))\in U_{F_n}.\label{UUU3}
\end{align}
Let $\mathrm{i}_0\colon \bigoplus_{i=1}^r\M_{p'_i}\to \M_{m'}$ and $\mathrm{i}_1\colon \bigoplus_{j=1}^s\M_{q'_j}\to \M_{m'}$ denote the inclusion maps, and let $\gamma_0, \gamma_1\colon \M_{m'}(\mathrm{C}[0,1])\to \M_{m'}$ denote the evaluation maps at $0$ and $1$, respectively. Note that 
\[
\Cu(\mathrm{i}_0\circ\delta_0)\circ\alpha=\Cu(\gamma_0\circ\mathrm{i}_B)\circ\alpha,\quad \Cu(\mathrm{i}_1\circ\delta_1)\circ\alpha=\Cu(\gamma_1\circ\mathrm{i}_B)\circ\alpha.
\]
Hence, by \eqref{UUU1}, \eqref{UUU2}, and \eqref{UUU3} we have
\begin{align*}
&(\Cu(\mathrm{i}_0\circ\delta_0)\circ\alpha,\Cu(\gamma_0\circ\phi))\in U_{F_n},\quad (\Cu(\mathrm{i}_0\circ\delta_0)\circ\alpha,\Cu(\mathrm{i}_0\circ\phi_0))\in U_{F_n},\\
&(\Cu(\mathrm{i}_1\circ\delta_1)\circ\alpha,\Cu(\gamma_1\circ\phi))\in U_{F_n}, \quad (\Cu(\mathrm{i}_1\circ\delta_1)\circ\alpha,\Cu(\mathrm{i}_1\circ\phi_1))\in U_{F_n}.
\end{align*}
It follows that
\begin{align}\label{CCppUU}
\begin{aligned}
&(\Cu(\gamma_0\circ\phi), \Cu(\mathrm{i}_0\circ\phi_0))\in U_{F_n}^2\subseteq U_{F_{n-1}},\\
&(\Cu(\gamma_1\circ\phi),\Cu(\mathrm{i}_1\circ\phi_1))\in U_{F_n}^2\subseteq U_{F_{n-1}}.
\end{aligned}
\end{align}

Since $\phi_0$ and $\phi_1$ have diagonal form there are permutation matrices $S_0,S_1\in \M_{m'}$ such that $\Ad(S_0)\circ\phi_0$ and $\Ad(S_1)\circ\phi_1$ have standard form. That is,
\begin{align*}
&S_0^*\phi_0(f)S_0=\diag(\Lambda_{\overline{\nu^{\phi_0}}}, \Lambda_{\overline{\omega^{\phi_0}}},f(\lambda_1^{\phi_0}), \cdots, f(\lambda_{\mu^{\phi_0}}^{\phi_0})),\\
&S_1^*\phi_1(f)S_1=\diag(\Lambda_{\overline{\nu^{\phi_1}}}, \Lambda_{\overline{\omega^{\phi_1}}},f(\lambda_1^{\phi_1}), \cdots, f(\lambda_{\mu^{\phi_1}}^{\phi_1})),
\end{align*}
where the tuples $\overline{\nu^{\phi_0}}$, $\overline{\omega^{\phi_0}}$, $\overline{\nu^{\phi_1}}$, and $\overline{\omega^{\phi_1}}$ have a component equal to zero. Also, since $\phi$ has standard form
\begin{align}\label{phi_matrix}
\phi(f)=\diag(\Lambda_{\overline{\nu^{\phi}}}, \Lambda_{\overline{\omega^{\phi}}},f(\lambda_1^{\phi}), \cdots, f(\lambda_{\mu^{\phi}}^{\phi})).
\end{align}
It follows now by \eqref{CCppUU}, as in the proof of Theorem \ref{Uniqueness}, that 
\begin{align*}
\mu^\phi=\mu^{\phi_0}=\mu^{\phi_1},\quad
\overline{\nu^\phi}=\overline{\nu^{\phi_0}}=\overline{\nu^{\phi_1}},\quad\overline{\omega^\phi}=\overline{\omega^{\phi_0}}
=\overline{\omega^{\phi_1}},
\end{align*}
\begin{align}\label{llambda}
|\lambda_i^\phi(0)-\lambda_i^{\phi_0}|\le 1/2^{n-2},\quad |\lambda_i^\phi(1)-\lambda_i^{\phi_1}|\le 1/2^{n-2},
\end{align}
for $i=1,2,\cdots,\mu$.

By Equation \eqref{llambda} for each $i=1,2,\cdots,\mu$ we can choose a continuous function $\lambda'_i\colon [0,1]\to [0,1]$ such that $\lambda_i'(0)=\lambda_i^{\phi_0}$, $\lambda_i'(1)=\lambda_i^{\phi_1}$, and 
\begin{align}\label{sup_lambda}
\sup_{t\in [0,1]}|\lambda_i^\phi(t)-\lambda_i'(t)|\le 1/2^{n-2}.
\end{align}
Let $U\colon [0,1]\to M_{m'}$ be a continuous path of unitaries such that $U(0)=S_0$ and $U(1)=S_1$. Let us show that the $\ast$-homomorphism $\phi'\colon A\to B$ defined by
\begin{align}\label{phi_prime}
\phi'(f)=U\diag(\Lambda_{\overline{\nu}}, \Lambda_{\overline{\omega}},f\circ\lambda'_1, \cdots, f\circ\lambda'_{\mu})U^*
\end{align} 
satisfies the conditions of the lemma. It sends elements of $A$ to $B$ since by construction 
\begin{align}\label{phiphi_prime}
 \phi'(f)(0)=\phi_0(f),\quad \phi'(f)(1)=\phi_1(f),
\end{align}
for all $f\in A$. 
It is left to show that $(\alpha, \phi')\in U_{n-3}$. In other words we need to show that
\begin{align*}
&\alpha(x_{i,j})=\Cu(\phi')(x_{i,j}),\\
&\alpha(y_{(k+1)/2^{n-3}})\le \Cu(y_{k/2^{n-3}}),\\
&\Cu(\phi')(y_{(k+1)/2^{n-3}})\le \alpha(y_{k/2^{n-3}}),
\end{align*}
for every $1\le i\le r$, $1\le j\le s$, and $0\le k\le 2^{n-3}-1$.

By equations \eqref{UUU1}, \eqref{UUU2}, and \eqref{UUU3} we have
\begin{align*}
\alpha(x_{i,j})=
\begin{cases}
\Cu(\phi)(x_{i,j})(t)&\text{if }t\in (0,1),\\
\Cu(\phi_0)(x_{i,j})(t)&\text{if }t\in \{0_1^B, 0_2^B, \cdots, 0_{r'}^B\},\\
\Cu(\phi_1)(x_{i,j})(t)&\text{if }t\in \{1_1^B, 1_2^B, \cdots, 1_{s'}^B\}.
\end{cases}
\end{align*}
Hence, by Lemma \ref{Cuphi} and the construction of $\phi'$ it follows that
\begin{align}\label{alphapp}
\alpha(x_{i,j})=\Cu(\phi')(x_{i,j}),
\end{align}
for every $1\le i\le r$ and $1\le j\le s$. 

For $t\in (0,1)$ and $1\le k\le 2^{n-3}-1$ we have
\begin{align*}
\alpha(y_{(k+1)/2^{n-3}})(t)&=\alpha(y_{(2k+2)/2^{n-2}})(t)\\
&\stackrel{\eqref{UUU1}}{\le} \Cu(\phi)(y_{(2k+1)/2^{n-2}})(t)\\
&\stackrel{\eqref{phi_matrix}}{=}|\{i : \lambda_i^\phi(t)>(2k+1)/2^{n-2}\}|+w_1^\phi\\
&\stackrel{\eqref{sup_lambda}}{\le} |\{i : \lambda'_i(t)>2k/2^{n-2}\}|+w_1^\phi\\
&\stackrel{\eqref{phi_prime}}{=}\Cu(\phi')(y_{k/2^{n-3}})(t),\\
\Cu(\phi')(y_{(k+1)/2^{n-3}})(t)&=\Cu(\phi')(y_{(2k+2)/2^{n-2}})(t)\\
&\stackrel{\eqref{phi_prime}}{=}|\{i : \lambda'_i(t)>(2k+2)/2^{n-2}\}|+w_1^\phi\\
&\stackrel{\eqref{sup_lambda}}{\le} |\{i : \lambda_i^\phi(t)>(2k+1)/2^{n-2}\}|+w_1^\phi\\
&\stackrel{\eqref{phi_matrix}}{\le} \Cu(\phi)(y_{(2k+1)/2^{n-2}})(t)\\
&\stackrel{\eqref{UUU1}}{\le} \alpha(y_{2k/2^{n-2}})(t)\\
&=\alpha(y_{k/2^{n-3}})(t).
\end{align*}
For $t\in \{0_1^B, 0_2^B, \cdots, 0_{r'}^B\}$ and $1\le k\le 2^{n-3}-1$ we have
\begin{align*}
\alpha(y_{(k+1)/2^{n-3}})(t)&=\alpha(y_{(2k+2)/2^{n-2}})(t)\\
&\stackrel{\eqref{UUU2}}{\le} \Cu(\phi_0)(y_{(2k+1)/2^{n-2}})(t)\\
&\stackrel{\eqref{phiphi_prime}}{=}\Cu(\phi')(y_{(2k+1)/2^{n-2}})(t)\\
&\le \Cu(\phi')(y_{k)/2^{n-3}})(t),\\
\Cu(\phi')(y_{(k+1)/2^{n-3}})(t)&\stackrel{\eqref{phiphi_prime}}{=}\Cu(\phi_0)(y_{(k+1)/2^{n-3}})(t)\\
&=\Cu(\phi_0)(y_{(2k+2)/2^{n-2}})(t)\\
&\stackrel{\eqref{UUU2}}{\le} \alpha(y_{(2k+1)/2^{n-2}})(t)\\
&\le \alpha(y_{k/2^{n-3}})(t).
\end{align*}
Similarly, for $t\in \{1_1^B, 1_2^B, \cdots, 1_{s'}^B\}$ and $1\le k\le 2^{n-3}-1$ 
\begin{align*}
\alpha(y_{(k+1)/2^{n-3}})(t)\le \Cu(\phi')(y_{k/2^{n-3}})(t),\quad \Cu(\phi')(y_{(k+1)/2^{n-3}})(t)\le \alpha(y_{k/2^{n-3}})(t).
\end{align*}
We have shown that,
\[
\alpha(y_{(k+1)/2^{n-3}})\le \Cu(\phi')(y_{k/2^{n-3}}),\quad \Cu(\phi')(y_{(k+1)/2^{n-3}})\le \alpha(y_{k/2^{n-3}}).
\]
This together with \eqref{alphapp} implies that
\[
(\Cu(\phi),\alpha)\in U_{F_{n-3}}.
\]
\end{proof}

\begin{proof}[Proof of Theorem \ref{Existence}]
Let $A=\mathrm{S}_m[\overline{p}, \overline{q},r,s]$ be a splitting interval algebra, and let $B=\varinjlim (B_i, \phi_{i,j})$ be an inductive limit of finite direct sums of splitting interval algebras. First let us show that the Theorem holds in the case that the $\ast$-homomorphisms $\phi_{i,j}\colon B_i\to B_j$ are unital and that $\alpha[1_A]=[1_B]$. 

Let $F$ be a finite subset of $\Cu(A)$. By Proposition \ref{basis} there is $N\ge 1$ such that $U_{F_{N-3}}\subseteq U_{F}$. Hence, it is enough to show that the theorem holds in the case that $F=F_{N-3}$. Let $x_{i,j}$, $1\le i\le r$, $1\le j\le s$, and $y_{l/2^N}$, $1\le l\le 2^N-1$, be the elements of $F_N$. We have
\begin{align}\label{xy1}
& \alpha(x_{i,j})+\alpha(x_{i',j'})=\alpha(x_{i,j'})+\alpha(x_{i',j}),\quad \alpha(x_{i,j})\le \alpha[1_A]=[1_B],\\
&\sum_{i=1}^rp_i\alpha(x_{i,j'})+\sum_{j=1}^sq_j\alpha(x_{i',j})=\alpha[1_A]+(2m-1)\alpha(x_{i',j'}),\label{xy2}\\
&\alpha(x_{i,1})\le \alpha(y_{l/2^N})+\alpha(x_{i,j}), \quad \alpha(y_{(l+1)/2^N})\ll \alpha(y_{l/2^N}),\quad \alpha(y_0)\le \alpha(x_{i,1}),\label{xy3}
\end{align}
for every $1\le i,i'\le r$, $1\le j,j'\le s$, and $1\le l\le 2^N-1$.

By Lemma \ref{compact_elements} there exist $k\ge 1$ and compact elements $z_{i,j}\in \Cu(B_k)$ such that $\Cu(\phi_{k,\infty})(z_{i,j})=\alpha(x_{i,j})$. Hence, by \eqref{xy1} and \eqref{xy2} we have
\begin{align*}
&\Cu(\phi_{k,\infty})(z_{i,j}+z_{i',j'})=\Cu(\phi_{k,\infty})(z_{i,j'}+z_{i',j}),\quad \Cu(\phi_{k,\infty})(z_{i,j})\le \Cu(\phi_{k,\infty})[1_{B_k}],\\
&\Cu(\phi_{k,\infty})\left(\sum_{i=1}^rp_iz_{i,j'}+\sum_{j=1}^sq_jz_{i',j}\right)=\Cu(\phi_{k,\infty})([1_{B_k}]+(2m-1)z_{i',j'}).
\end{align*}
By (i) and (ii) of Lemma \ref{compact_elements} there exists $n\ge k$ such that
\begin{align*}
&\Cu(\phi_{k,n})(z_{i,j}+z_{i',j'})=\Cu(\phi_{k,n})(z_{i,j'}+z_{i',j}),\quad \Cu(\phi_{k,n})(z_{i,j})\le \Cu(\phi_{k,n})[1_{B_k}],
\end{align*}
and
\begin{align*}
\Cu(\phi_{k,n})\left(\sum_{i=1}^rp_iz_{i,j'}+\sum_{j=1}^sq_jz_{i',j}\right)&=\Cu(\phi_{k,n})([1_{B_k}]+(2m-1)z_{i',j'})\\
&=[1_{B_n}]+\Cu(\phi_{k,n})((2m-1)z_{i',j'}),
\end{align*}
for every $1\le i,i'\le r$ and $1\le j,j'\le s$. Hence, by replacing the elements $z_{i,j}$ by $\Cu(\phi_{k,n})(z_{i,j})$ we may assume that 
\begin{align}\label{zetas}
\begin{aligned}
&z_{i,j}+z_{i',j'}=z_{i,j'}+z_{i',j},\quad z_{i,j}\le [1_{B_k}],\\
&\sum_{i=1}^rp_iz_{i,j'}+\sum_{j=1}^sq_jz_{i',j}=[1_{B_k}]+(2m-1)z_{i',j'}.
\end{aligned}
\end{align}

By (i) of Proposition \ref{inductive_limits} for each $1\le l\le 2^N-1$ there are elements $t_{l/2^N}^{(n)}\in \Cu(B_n)$, $n=1,2,\cdots$, such that 
\[
\Cu(\phi_{n, n+1})\left(t_{l/2^N}^{(n)}\right)\ll t_{l/2^N}^{(n)},\quad \sup_n\Cu(\phi_{n,\infty})\left(t_{l/2^N}^{(n)}\right)=\alpha(y_{l/2^N}).
\]
Hence, by \eqref{xy3} and the definition of the relation $\ll$ there exists $n\ge 1$ such that
\begin{align}\label{CaC}
\begin{aligned}
&\Cu(\phi_{k,\infty})(z_{i,1})\le \Cu(\phi_{n,\infty})\left(t_{l/2^N}^{(n)}\right)+\Cu(\phi_{k,\infty})(z_{i,j}),\\
&\alpha(y_{(l+1)/2^N})\le \Cu(\phi_{n,\infty})\left(t_{l/2^N}^{(n)}\right)\le \alpha(y_{l/2^N}),\\
&\Cu(\phi_{n+1,\infty})\left(t_{0}^{(n+1)}\right)\le \Cu(\phi_{k,\infty})(z_{i,1}),
\end{aligned}
\end{align} 
for every $1\le i\le r$, $1\le j\le s$, and $1\le l\le 2^N-1$. Moreover, $n$ may be taken such that 
\begin{align}\label{cpcpt}
\Cu(\phi_{n+1,\infty})\left(t^{(n+1)}_{(l+1)/2^N}\right)\le \Cu(\phi_{n,\infty})\left(t^{(n)}_{l/2^N}\right),
\end{align}
for every $1\le l\le 2^N-1$. Since 
\[
z_{i,1}\ll z_{i,1},\quad \Cu(\phi_{n,n+1})\left(t_0^{(n)}\right)\ll t_0^{(n+1)},\quad \Cu(\phi_{n,n+1})\left(t^{(n)}_{(l+1)/2^N}\right)\ll t^{(n+1)}_{(l+1)/2^N},
\]
then by (ii) of Lemma \ref{inductive_limits} applied to the first and third inequality of \eqref{CaC} and to the inequality \eqref{cpcpt} there exists $k'$ such that 
\begin{align}\label{zetas_tes}
\begin{aligned}
&\Cu(\phi_{k,k'})(z_{i,1})\le \Cu(\phi_{n,k'})\left(t_{l/2^N}^{(n)}\right)+\Cu(\phi_{k,k'})(z_{i,j})\\
&\Cu(\phi_{n,k'})\left(t_{0}^{(n)}\right)\le \Cu(\phi_{k,k'})(z_{i,1}),\\
&\Cu(\phi_{n,k'})\left(t^{(n)}_{(l+1)/2^N}\right)\le \Cu(\phi_{n,k'})\left(t^{(n)}_{l/2^N}\right)
\end{aligned}
\end{align} 
for every $1\le i\le r$, $1\le j\le s$, and $1\le l\le 2^N-1$. Set $\Cu(\phi_{k,k'})(z_{i,1})=\tilde{z}_{i,j}$ and $\Cu(\phi_{n,k'})\left(t^{(n)}_{l/2^N}\right)=\tilde{t}_{l/2^N}$. Then, by \eqref{zetas}, the second inequality of \eqref{CaC}, and \eqref{zetas_tes} we have
\begin{align*}
&\tilde{z}_{i,j}+\tilde{z}_{i',j'}=\tilde{z}_{i,j'}+\tilde{z}_{i',j},\quad \tilde{z}_{i,j}\le [1_{B_k}],\\
&\sum_{i=1}^rp_i\tilde{z}_{i,j'}+\sum_{j=1}^sq_j\tilde{z}_{i',j}=[1_{B_k}]+(2m-1)\tilde{z}_{i',j'},\\
&\tilde{z}_{i,1}\le \tilde{t}_{l/2^N}+\tilde{z}_{i,j},\quad \tilde{t}_{(l+1)/2^N}\le \tilde{t}_{l/2^N}\le \tilde{z}_{i,1},
\end{align*}
for every $1\le i,i'\le r$, $1\le j, j'\le s$, and $1\le l\le 2^N-1$. Let $\tilde\alpha\colon F_N\to \Cu(B_{k'})$ be the map defined by $\alpha(x_{i,j})=\tilde{z}_{i,j}$ and $\alpha(y_{l/2^N})=\tilde{t}_{l/2^N}$. Then, it follows by the preceding equations that $\tilde\alpha$ satisfies conditions (i) to (vii) of Lemma \ref{app-lif1}. Hence, by Lemma \ref{app-lif2} there exists a $\ast$-homomorphism $\psi\colon A\to B_{k'}$ such that $(\Cu(\psi), \tilde\alpha)\in U_{F_{N-2}}$. It follows that
\begin{align}\label{CPU}
(\Cu(\phi_{k',\infty}\circ\psi), \Cu(\phi_{k',\infty})\circ\tilde\alpha)\in U_{F_{N-2}}.
\end{align}
By the second inequality of \eqref{CaC} we have
\begin{align*}
\alpha(y_{(l+1)/2^N})\le \Cu(\phi_{k',\infty})\circ\tilde{\alpha}(y_{l/2^N})\le\alpha(y_{l/2^N}),
\end{align*}
for every $1\le l\le 2^{N}-1$. This implies by the definition of the entourage $U_{F_N}$ that
\begin{align*}
(\alpha, \Cu(\phi_{k',\infty})\circ\tilde{\alpha})\in U_{F_N}\subseteq U_{F_{N-2}}.
\end{align*}
This together with \eqref{CPU} imply that
\[
(\alpha, \Cu(\phi_{k',\infty}\circ\psi))\in U_{F_{N-2}}^2\subseteq U_{F_{N-3}}.
\]
The theorem follows by taking $\phi=\phi_{k',\infty}\circ\psi$.

Let us consider the general case. Let $F\subset \Cu(A)$ be a finite subset, and let $\alpha\colon \Cu(A)\to \Cu(B)$ be such that $\alpha[1_A]\le [s_B]$. Since $\alpha[1_A]$ is a compact element of $\Cu(B)$ and $B$ has stable rank one, there is a projection $P\in B$ such that $\alpha[1_A]=[P]$. Since projections lift there exist $i\ge 1$ and a projection $Q\in B_i$ such that $\phi_{i,\infty}(Q)=P$. The C*-algebras $B_i'=\phi_{i,j}(Q)B_j\phi_{i,j}(Q)$ are isomorphic to direct sum of splitting interval algebras. In addition, we have
\[
PBP=\varinjlim (B_i',\phi_{i,j}|_{B_i'}).
\]
Note that the $\ast$-homomorphisms $\phi_{i,j}|_{B_i'}\colon B_i'\to B_j'$ are unital. Since $\alpha[1_A]=[P]$ the image of the morphism $\alpha$ is contained in $\Cu(PBP)$ which is a subsemigroup of $\Cu(B)$. Hence, applying the theorem in the case that the codomain algebra is an inductive limit of direct sums of splitting interval algebras with unital $\ast$-homomorphisms there exist a $\ast$-homomorphism $\phi\colon A\to PBP\subset B$ such that $(\alpha, \Cu(\phi))\in U_{F}$. This concludes the proof of the theorem. 
\end{proof}
\section{Proof of Theorem \ref{homomorphism}}

\begin{proof}[Proof of Theorem \ref{homomorphism}]
By (iv) of Proposition 5 of \cite{Ciuperca-Elliott-Santiago} we may assume that $A$ is a finite direct sum of spitting interval algebras. Furthermore, by the proof of (iii) of Proposition 5 of \cite{Ciuperca-Elliott-Santiago} we may assume that $A$ is a splitting interval algebra. This follows since for every projection $p\in B$ the hereditary subalgebra $pBp$ can be written as a sequential inductive limit of finite direct sums of splitting interval algebras.

Let $A=\mathrm{S}[\overline{p}, \overline{q},r,s]$ be a splitting interval algebra and let $B$ be an inductive limit of finite direct sums of splitting interval algebras. Since $A$ is separable we can choose finite sets $G_i\subset A$, $i=1,2,\cdots$, such that $A=\overline{\bigcup_{i=1}^\infty G_i}$,  and $G_i\subset G_{i+1}$, for $i\ge 1$. By Theorem \ref{Uniqueness} for each $i\ge 1$ there exist a finite subset $F\in \Cu(A)$ such that: if $\phi,\psi\colon A\to B$ are $\ast$-homomorphisms such that $(\Cu(\phi),\Cu(\psi))\in U_F$, then there is a unitary $u$ in the unitization of $B$ such that
\begin{align}\label{phiupsiu}
\|\phi(a)-u^*\psi(a)u\|<\frac{1}{2^i},
\end{align}
for all $a\in G_i$. Since the entourages $U_{F_i}$, $i=1,2,\cdots$, are a basis for the uniform structure $\mathcal{U}_{A,B}$ (see Subsection \ref{Uniform-structure}) there exists $k_i\ge 1$ such that $U_{k_i}\in U_{F}$. Therefore, if $(\phi, \psi)\in U_{F_{k_i}}$ then \eqref{phiupsiu} holds.

Let $\alpha\colon \Cu(A)\to \Cu(B)$ be a Cuntz semigroup morphism such that $\alpha[1_A]\le [s_B]$, where $s_B$ is an strictly positive element of $B$. By Theorem \ref{Existence} for each $n\ge 1$ there exist a $\ast$-homomorphisms $\phi_n\colon A\to B$ such that $(\alpha, \Cu(\phi_n))\in U_{F_n}$. This implies by \eqref{uniform_1} that
\[
(\Cu(\phi_m),\Cu(\phi_n))\in U_{F_{N+1}}^2\subseteq U_{F_N}
\]
for every $m,n>N$. Thus, we can choose a sequence of natural numbers $n_i$, $i=1,2,\cdots$, such that $(\Cu(\phi_{n_i}), \Cu(\phi_{n_{i+1}}))\in U_{k_i}$ for all $i\ge 1$. By the choice of numbers $k_i$, $i=1,2,\cdots$, there are unitaries $u_i\in \widetilde{B}$ such that
\begin{align}\label{phiuiphiui}
\|\phi_{n_i}(a)-u_i^*\phi_{n_{i+1}}(a)u_i\|<\frac{1}{2^i},
\end{align}
for every $a\in G_i$, and $i\ge 1$. Set
\[
\Ad(u_{i-1}\cdots u_2u_1)\circ\phi_{n_i}=\psi_i.
\]
Then, by \eqref{phiuiphiui} 
\[
\|\psi_i(a)-\psi_{i+1}(a)\|<\frac{1}{2^i},
\]
for every $a\in G_i$, and $i\ge 1$. It follows that
the sequence $\psi_i(f)$, $i=1,2\cdots$, is Cauchy for every $f\in \bigcup_{i=1}^\infty G_i$. Therefore, the limit 
\begin{align}\label{plimp}
\phi(f)=\lim_{i\to \infty} \psi_i(f),
\end{align}
exists for every $f\in \bigcup_{i=1}^\infty G_i$. Since the set $ \bigcup_{i=1}^\infty G_i$ is dense in $A$ and the map $\phi\colon \bigcup_{i=1}^\infty G_i\to B$ is the pointwise limit of $\ast$-homomorphisms, it extends to a $\ast$-homomorphism $\phi\colon A\to B$. Let us show that $\Cu(\phi)=\alpha$.

Let $G$ be the finite subset of $A$ defined in \eqref{G}. Without loss of generality we may assume that $G$ is contained in $\bigcup_{i=1}^\infty G_i$. By \eqref{plimp} for each $n\ge 1$ there exists $k\ge n$ such that
\begin{align}\label{pppp_1}
\|\phi(f)-\psi_k(f)\|<\frac{1}{2^n},
\end{align}
for every $f\in G$. Hence, by (i) of Lemma \ref{lemma_uniqueness} we have $(\Cu(\phi), \Cu(\psi_k))\in U_{F_n}$. It follows that $(\Cu(\phi), \Cu(\phi_k))\in U_{F_n}$ since the $\ast$-homomorphisms $\phi_k$ and $\psi_k$ are unitarily equivalent. We have $(\Cu(\phi_k),\alpha)\in U_{F_k}\subseteq U_{F_n}$. Therefore,
\[
(\Cu(\phi), \alpha)\in U^2_{F_n}\subseteq U_{F_{n-1}}.
\]
Since $n$ is arbitrary and the entourages $U_{F_n}$, $n=1,2,\cdots$, are a basis for the uniform structure on the set of Cuntz semigroup morphisms from $\Cu(A)$ to $\Cu(B)$ we conclude that $\alpha=\Cu(\phi)$. 
 
\end{proof}

The proof of Corollary \ref{isomorphism} is the same as the proof of Corollary 1 of \cite{Ciuperca-Elliott-Santiago}.

\end{document}